\documentclass[11pt,a4paper,reqno,twoside]{amsart}

\usepackage{amssymb}
\usepackage{latexsym}
\usepackage{amsmath}
\usepackage{amsthm}
\usepackage{enumitem,cite}
\usepackage{tikz}
\usepackage{pgfplots}
\setlist[enumerate]{leftmargin=1.4cm}
\usepackage{bm}
\usepackage{subcaption}
\usepackage[top=2.8cm, left=2.65cm, right=2.65cm, bottom=2.8cm]{geometry}

\definecolor{blueish}{rgb}{0.0, 0.53, 0.74}
\definecolor{blush}{rgb}{0.87, 0.36, 0.51}

\numberwithin{equation}{section}
\numberwithin{figure}{section}

\theoremstyle{definition}
\newtheorem{theorem}{Theorem}[section]
\newtheorem{corollary}[theorem]{Corollary}
\newtheorem{proposition}[theorem]{Proposition}
\newtheorem{definition}[theorem]{Definition}
\newtheorem{example}[theorem]{Example}
\newtheorem{notation}[theorem]{Notation}
\newtheorem{remark}[theorem]{Remark}
\newtheorem{lemma}[theorem]{Lemma}

\newcommand\qbin[3]{\textnormal{bin}_{#3}(#1,#2)}

\newcommand\bbq[1]{\bm{b}_q(#1)}

\newcommand{\numberset}{\mathbb}
\newcommand{\N}{\numberset{N}}
\newcommand{\Z}{\numberset{Z}}

\newcommand{\R}{\numberset{R}}

\newcommand{\F}{\numberset{F}}

\newcommand{\ellmax}{\ell_{\text{max}}}

\newcommand{\mV}{\mathcal{V}}
\newcommand{\mW}{\mathcal{W}}

\newcommand{\mL}{\mathcal{L}}

\newcommand{\mC}{\mathcal{C}}
\newcommand{\mA}{\mathcal{A}}
\newcommand{\mB}{\mathcal{B}}

\newcommand{\mF}{\mathcal{F}}
\newcommand{\mE}{\mathcal{E}}
\newcommand{\rk}{\textnormal{rk}}

\newcommand{\mU}{\mathcal{U}}

\newcommand{\mat}{\F_q^{n \times m}}
\renewcommand{\longrightarrow}{\to}

\newcommand{\drk}{d}

\newcommand{\MRD}{{\textnormal{MRD}}}

\usepackage{xcolor}

\newcommand*{\myproofname}{Proof of the claim}

\newtheorem{problem}{Problem}

\usepackage{amsaddr}
\usepackage{hyperref}
\usepackage{graphicx}

\usepackage{setspace}
\usepackage{calrsfs}

\title[Common Complements of Linear Subspaces and the Sparseness of
MRD Codes]{Common Complements of Linear Subspaces \\ and the Sparseness of MRD Codes}

\author{Anina Gruica and Alberto Ravagnani}

\email{a.gruica@tue.nl,  a.ravagnani@tue.nl}
\subjclass[2010]{11T71, 05A16}

\address{Department of Mathematics and Computer Science \\
Eindhoven University of Technology, 
the Netherlands}

\thanks{The authors 
were partially supported by the Dutch Research Council through grant OCENW.KLEIN.539.}

\date{}

\begin{document}

\vspace*{-0.3cm}

\maketitle

\thispagestyle{empty}

\begin{abstract}
Motivated by applications to the theory of rank-metric codes, we study the problem of estimating the number of common complements of a family of subspaces over a finite field in terms of the cardinality of the family and its intersection structure. We derive upper and lower bounds for this number, along with their asymptotic versions as the field size tends to infinity. We then use these bounds to 
describe the general behaviour of common complements with respect to sparseness and density, showing that
the decisive property is
whether or not the number of spaces to be complemented is negligible with respect to the field size.
By specializing our results to matrix  spaces,
we obtain upper and lower bounds for the number of MRD codes in the rank metric. In particular, 
we
answer an open question 
in coding theory, proving 
that MRD codes are sparse for all parameter sets as the field size grows, with only very few exceptions.
We also investigate the  density of MRD codes as their number of columns tends to infinity,
obtaining a new asymptotic bound. Using properties of the Euler function from number theory, we then show that our bound improves on known results
for most parameter sets. 
We conclude the paper by establishing general
structural properties of the density function of rank-metric codes.
\end{abstract}

\bigskip

\bigskip

\section*{Introduction}

 A \textit{rank-metric code} is a linear space of matrices over a finite field $\F_q$ in which every non-zero matrix has rank bounded from below by an integer $d$ (called the \textit{minimum distance} of the code). Originally introduced by  Delsarte for combinatorial interest~\cite{delsarte1978bilinear}, in the last few decades rank-metric codes have 
been extensively studied in connection with various applications in information technology~\cite{koetter2008coding,gabidulin,SKK,roth1991maximum} and several areas of pure and applied mathematics,
including combinatorial designs,
rook theory, semifields, polymatroids and linear sets; see~\cite{delsarte1978bilinear,sheekey2020new,braun2016existence,lewis2020rook,gorla2018rankq,schmidt2020quadratic,csajbok2017maximum} among many others.

An open question in coding theory asks to compute the asymptotic density of rank-metric codes having maximum dimension, also known as
\textit{maximum rank distance} (MRD) \textit{codes}; see for example~\cite{byrne2020partition,antrobus2019maximal}.
More in detail, one fixes a value for the minimum distance and attempts to compute the asymptotics, as $q \to +\infty$,
of the proportion of MRD codes having 
that distance 
within the set of codes sharing the same 
dimension.
To date, three independent approaches have been developed in the attempt to solve this problem, based on enumerative combinatorics, the theory of spectrum-free matrices, and semifields; see~\cite{antrobus2019maximal,byrne2020partition,gluesing2020sparseness}. 
All these 
techniques 
show that $\F_q$-linear
MRD codes are not dense within the set of codes having a certain dimension. This is in sharp contrast with the behavior of MDS codes in the Hamming metric and of $\F_{q^m}$-linear MRD codes, which are natural analogues of $\F_q$-linear MRD codes and are instead 
dense as the field size tends to infinity~\cite{byrne2020partition,neri2018genericity}.

In this paper, we reinterpret the above question as a broader problem intersecting combinatorial geometry and extremal combinatorics, which is interesting in its own right. 
More precisely, we study the problem of estimating the number of complements shared by a family of  subspaces
of $\F_q^N$, say $\mA$, all of which have the same codimension $k$. The bounds that we derive take into account the \textit{intersection structure} of the spaces in $\mA$ (i.e., how many subspace pairs intersect in a given dimension), as well as the cardinality of $\mA$. The question of estimating the density of MRD codes turns out to be a very special instance of this general problem.

Our strategy to obtain upper and lower bounds for the number of common complements of the spaces in $\mA$ relies on the rigidity of certain graphs constructed from a linear lattice. We 
introduce a simple notion of regularity
of a bipartite graph with respect to~a function defined on its left-vertices, which we call an \textit{association}. This extends the concept of \textit{left-regularity} and defines a set of numerical parameters of the underlying graph.
We then describe the aforementioned complements as the isolated right-vertices of such a regular
bipartite graph, estimating their number in terms of the fundamental graph's parameters.
In turn, these parameters
can be computed using classical  methods from the theory of \textit{critical problems} in combinatorial geometry.

Of particular interest for us are the asymptotic versions of these bounds which, under certain assumptions, lead to the following general behavior of the common complements
with respect to sparseness/density.
If the cardinality of~$\mA$ is negligible with respect to the field size~$q$, then almost all $k$-subspaces of $\F_q^N$ are common complements of the spaces in $\mA$; moreover,~the proportion of non-common complements is in~$O(|\mA|/q)$ as $q \to +\infty$. Vice versa,
if the cardinality of~$\mA$ is preponderant with respect to the field size $q$, then the 
common complements are sparse (precise asymptotic estimates will be given). In our asymptotic analysis, we find particularly useful the notion of an \textit{asymptotic partial spread}, which we propose as the asymptotic analogue of the classical and homonymous definition from discrete geometry. 

In the second part of the paper we turn to the theory of rank-metric codes, specializing our results to matrix spaces over $\F_q$. 
Our main result is an upper bound on the number of MRD codes with given parameters; see Theorem~\ref{thm:mrdboundcc}. 
We also prove that the density of $n \times m$ MRD codes of minimum distance $d$ is in $${O\left(q^{-(d-1)(n-d+1)+1}\right)} \quad \textnormal{ as $q \to +\infty$;}$$
see Theorem~\ref{sparseness} for a precise statement.
This shows that MRD codes are  very sparse, unless $d=1$ or $n=d=2$, answering the question stated at the beginning of this introduction.

The third part of the paper concentrates on the asymptotic density of $n \times m$ MRD codes as $m \to +\infty$. We apply the graph theory machinery described above and obtain an upper bound on the limit superior of the density of these codes. Our estimates involve the 
\textit{Euler function}~$\phi$ from the theory of $q$-series. In fact, with the aid of Euler's Pentagonal Number Theorem we show that our asymptotic bounds improve on known results for most parameter sets. The question of determining whether or not MRD codes are sparse for $m$ large remains open.

In the last section of the paper we investigate some general properties of density functions in the rank metric, without restricting to MRD codes necessarily. This also gives us the chance to reinterpret known results from a new perspective.

\bigskip

\noindent \textbf{Outline.}
The remainder of the paper is organized as follows.
In Section~\ref{sec:prel} we illustrate the problems we study and introduce the relevant terminology.
Section~\ref{sec:counting}
contains preliminary formulas on linear spaces and tuples of functionals, which we will need repeatedly throughout the paper.
In Section~\ref{sec:bounds} we present our main results, deriving upper and lower bounds for the number of common complements of a family of subspaces using a graph theory approach. The asymptotic versions of these bounds are obtained in Section~\ref{sec:asy}.
We study the density function of MRD codes (and sometimes of more general rank-metric codes) for $q \to+\infty$ and $m\to+\infty$ 
in Sections~\ref{sec:rankmetric}
and~\ref{sec:m}, respectively.
Finally, structural properties of the density functions of rank-metric codes are established in Section~\ref{sec:prop}.

\bigskip

\section{Problem Formulation} \label{sec:prel}

In this section we recall some concepts from combinatorial geometry and state the main problems studied throughout the paper, illustrating their connection with the theory of rank-metric codes.
In the sequel, $q$ denotes a prime power and $\F_q$ is the finite field of $q$ elements. We let $$\qbin{a}{b}{q}= \prod_{i=0}^{b-1} \; \frac{q^a-q^i}{q^b-q^i}$$
be the $q$-binomial coefficient of integers $a \ge b \ge 0$; see e.g.~\cite{stanley2011enumerative}. It is well-known that $\qbin{a}{b}{q}$ counts the number of $b$-subspaces of an $a$-space over~$\F_q$.

\begin{definition}
Let $X$ be a vector space over $\F_q$ and let 
$W \le X$ be a subspace.
 A \textbf{complement} of $W$ in $X$ is a subspace 
 $W' \le X$ with $W \oplus W' = X$, i.e.,
 a complement of $W$ in the lattice of subspaces of $X$ (we denote by
 ``$\le$'' the inclusion relation of linear spaces).
 \end{definition}

This paper focuses on the problem of estimating the number of complements shared by a collection of subspaces. A strong motivation to study this problem comes
from the theory of rank-metric codes, as we will explain shortly.

\begin{problem} \label{pb:a}
Let $X$ be a vector space of finite dimension $N \ge 3$ over $\F_q$ and let $1 \le k \le N-1$ be an integer.
Let $\mA$ be a non-empty collection of subspaces of $X$, all of which have codimension~$k$. Give upper and lower bounds 
for the number of common complements in~$X$ of the spaces in~$\mA$.
\end{problem}

When studying Problem~\ref{pb:a}, we take into account structural properties of the set $\mA$ of combinatorial flavor, as we will explain later. 
In this paper we also investigate
the
asymptotic version of Problem~\ref{pb:a} as the field size tends to infinity, which can be stated as follows.

\begin{problem} \label{pb:a1}
Let $Q$ be the set of prime powers and let $(X_q)_{q \in Q}$ be a sequence of vector spaces, all of which have the same dimension $N \ge 3$ over $\F_q$.
Let $1 \le k \le N-1$ be an integer and let $(\mA_q)_{q \in Q}$ be a sequence of non-empty collections of linear spaces,
all of which have codimension~$k$,
with $A_q \le X_q$ for all $q \in Q$ and all $A_q \in \mA_q$. Determine the asymptotic behavior as $q \to +\infty$ of the ratio $|\mF_q|/\qbin{N}{k}{q}$, where 
$\mF_q$ is the collection of $k$-subspaces $W_q \le X_q$ that intersect some space in $\mA_q$ non-trivially. 
\end{problem}

Following the notation of Problem~\ref{pb:a1},
we say that the family of subspaces $(\mF_q)_{q \in Q}$  is \textbf{sparse} if $\lim_{q \to +\infty}\, |\mF_q|/\qbin{N}{k}{q} =0$ and \textbf{dense} if $\lim_{q \to +\infty}\, |\mF_q|/\qbin{N}{k}{q} =1$.

The asymptotic version of Problem~\ref{pb:a}, which we will study in Section~\ref{sec:asy}, is closely connected to an open question in coding theory on the density of MRD codes. 
We now briefly review 
some concepts from coding theory and 
explain this connection more in detail.

In the sequel, $m$ and $n$ denote integers with $m \ge n \ge 2$ and $\mat$ is the space of 
$n \times m$ matrices with entries in~$\F_q$.

\begin{definition}
A (\textbf{linear rank-metric}) \textbf{code} is a non-zero $\F_q$-linear subspace $\mC \le \mat$. Its \textbf{minimum distance}
is the integer $$\drk(\mC):=\min\{\rk(M) \mid M \in \mC, \; M \neq 0\}.$$ 
\end{definition}

A rank-metric code cannot have large dimension and minimum distance simultaneously.
The trade-off between these quantities is captured by the following result of Delsarte.

\begin{theorem}[Singleton-like Bound; see \cite{delsarte1978bilinear}]
\label{thm:slb}
Let $\mC \le \mat$ be a non-zero rank-metric code. We have
\begin{equation*} \label{singletonlikebound}
    \dim(\mC) \le m(n-\drk(\mC)+1).
\end{equation*}
\end{theorem}

The most studied rank-metric codes are those having the maximum possible dimension allowed by their minimum distance.

\begin{definition}
A code $\mC \le \mat$ 
is called \textbf{maximum rank distance}  (\textbf{MRD} in short)
if it attains the bound 
of Theorem~\ref{thm:slb} with equality.
\end{definition}

The coding theory problem we are interested in asks to compute the asymptotics, as the field size tends to infinity, of the proportion of MRD codes among all codes with a given dimension.
More formally (and more generally), we propose the following terminology.

\begin{definition} \label{def:delta}
For $1 \le k \le mn$
and $1 \le d \le n$, let $$\delta_q(n \times m, k, d):= \frac{|\{\mC \le \mat \mid \dim(\mC)=k, \; \drk(\mC) \ge d\}|}{\qbin{mn}{k}{q}}$$
denote the \textbf{density} (\textbf{function}) of rank-metric codes with minimum distance at least $d$ among all $k$-dimensional codes. Their 
\textbf{asymptotic density} is instead
$\lim_{q \to + \infty} \delta_q(n \times m, k, d)$, when the limit exists.
\end{definition}

The following is currently an open question in coding theory.

\begin{problem} \label{pb:b}
Compute $\lim_{q \to +\infty} \delta_q(n \times m,
m(n-d+1),d)$ for all 
$1 \le d \le n$, when it exists.
\end{problem}

We are also interested in determining the asymptotic density of MRD codes as their number of columns tends to infinity, which is another open problem.

\begin{problem} \label{pb:b1}
Compute $\lim_{m \to +\infty} \delta_q(n \times m,
m(n-d+1),d)$ for all 
$1 \le d \le n$, when
it exists.
\end{problem}

In this paper, we regard the above two problems
as special 
instances of Problem~\ref{pb:a1}.
This allows us to solve Problem~\ref{pb:b} and to make progress on Problem~\ref{pb:b1} in Sections~\ref{sec:rankmetric} and~\ref{sec:m}, respectively. In those sections, 
we will also survey the literature centered around these problems
and briefly describe the three approaches 
that have been explored so far to solve them.

In the following remark we describe more in detail the connection between Problems~\ref{pb:a}, \ref{pb:b}, and~\ref{pb:b1}.
This will be needed in later sections.

\begin{remark} \label{rem:expl}
Consider the collection $\mU$ of subspaces $U \le \F_q^n$ with $\dim(U)=d-1$. For any~$U \in \mU$, denote by $\smash{\mat(U)}$ the set of all matrices 
whose column-space is contained in $U$. It is easy to see that $\smash{\mat(U)}$ is a linear space of dimension $m(d-1)$ and that every element of $\smash{\mat(U)}$ has rank smaller or equal to $d-1$ for all $U \in \mU$. Finally, let $\smash{\mA=\{\mat(U) \mid U \in \mU\}}$. By definition, the common complements of the spaces in $\mA$ are the rank-metric codes $\smash{\mC \le \mat}$ that do not contain any matrix of rank smaller equal to $d-1$ and with dimension
$\dim(\mC)=mn-m(d-1)=m(n-d+1)$. This means that the common complements of the spaces in $\mA$ are exactly the MRD codes of minimum distance~$d$ in $\mat$. This interpretation of MRD codes as common complements of linear spaces will be crucial in our approach.
\end{remark}

We conclude this section by stating a fifth problem considered in this paper.
Although our main focus is not on this question, its solution will facilitate 
the study of Problem~\ref{pb:a1}. We need the following terminology.

 \begin{definition}
 Let $X$ be a vector space over $\F_q$. 
 A \textbf{cone} in $X$ is a non-empty subset $K \subseteq X$ with $\lambda v \in K$ for all $v \in K$ and $\lambda \in \F_q$.

Let $S \subseteq X$ be a set with $0 \in S$. We say that $W \le X$
\textbf{distinguishes} $S$ if 
  $W \cap S = \{0\}$.
 If this is not the case, then we say that 
 $W$ \textbf{intersects} $S$ and write $W \triangleright S$.
\end{definition}

\begin{problem} \label{pb:c}
Let $X$ be a vector space of finite dimension $N \ge 3$ over $\F_q$ and let $K \subseteq X$ be a cone. For $1 \le k \le N$, give upper and lower bounds for the number of $k$-spaces $W \le X$ 
that distinguish $K$.
\end{problem}

In the notation of Problem~\ref{pb:c}, the task of computing the \textit{exact} number of
$k$-spaces $W$ distinguishing $K$ is known to be difficult in general.
This is the celebrated \textit{critical problem} for combinatorial geometries, proposed
by Crapo and Rota in~\cite[Chapter 16]{crapo1970foundations}. Its solution heavily depends on the combinatorics of the cone $K$, in a precise lattice theory sense; see~\cite{dowling1971codes,zaslavsky1987mobius,kung1996critical,ravagnani2019whitney}.
There is also a very nice connection between 
Problem~\ref{pb:c} and \textit{Sperner theory}, which indirectly provides a partial solution for it. We will comment on this in Remark~\ref{rmk:sperner}.

\begin{remark}\label{rmk:a_vs_d1}
Since the union of linear subspaces is clearly a cone, at a first glance Problem~\ref{pb:c} might be seen as a special instance of Problem~\ref{pb:a}.
Our approach to the latter problem however takes into account information that gets lost when replacing a collection of subspaces with their union.
More precisely, the bounds that we will derive take into account 
the cardinality of $\mA$ (in the notation of Problem~\ref{pb:a}) as well as their \textit{intersection structure}, i.e., the number of subspace pairs intersecting in a given dimension.
Both these pieces of information are lost when replacing $\mA$ with $\bigcup \mA$. For this reason, in this paper Problems~\ref{pb:a} and~\ref{pb:c} are treated as very different questions. We will return to this discussion in Remark~\ref{rmk:a_vs_d2} and Example~\ref{exe:a_vs_d}. 
\end{remark}

\bigskip

\section{Counting Linear Spaces and Functionals}
\label{sec:counting}

The goal of this section is to provide a combinatorial interpretation for the following expression, which will be used repeatedly throughout the paper to derive bounds and their asymptotic versions.

\begin{notation}\label{notat:nu}
For a prime power $q$ and non-negative integers
$N$, $k$ and $\ell$ with $N \ge 3$, $k < N$ and $N-2k \le \ell \le N-k$, let
\begin{align}
    \nu_q(N,k,\ell) &:= \qbin{N}{k}{q} - 2q^{k(N-k)} + q^{(2k-N+\ell)(N-k)}\prod_{i=\ell}^{N-k-1}(q^{N-k}-q^i),
\end{align}
where throughout this paper a product over an empty index set is 1 by convention.
\end{notation}

We will show that 
$\nu_q(N,k,\ell)$ counts the number of $k$-subspaces of an $N$-space over $\F_q$ having a particular property. More precisely, the following holds.

%%%%%%%%%%%%%%%%%%%%%%%%%%%%%%%%%%
\begin{theorem} \label{thm:nu}
Let $N$, $k$ and $\ell$ be as in Notation~\ref{notat:nu}.
Let $X$ be an $N$-space over $\F_q$ and let $A,A',B,B' \le X$ be $(N-k)$-subspaces with $\ell=\dim(A \cap B) = \dim(A' \cap B')$. We have
\begin{equation*}
|\{W \le X \mid \dim(W)=k, \, 
W \triangleright A, \, W \triangleright B
\}| \, 
= \, |\{W \le X \mid \dim(W)=k, \, 
W \triangleright A', \, W \triangleright B'\}|.    
\end{equation*}
In words, the number of $k$-spaces
$W \le X$
intersecting $(N-k)$-spaces $A,B \le X$ only depends on $\ell = \dim(A \cap B)$. Moreover, this
number is precisely $\nu_q(N,k,\ell)$.
\end{theorem}
%%%%%%%%%%%%%%%%%%%%%%%%%%%%%%%%%%

Theorem~\ref{thm:nu} will be established after a series of preliminary results on linear functionals, which are natural objects in the theory of 
\textit{critical problems}~\cite{kung1996critical,crapo1970foundations}. While there are more direct approaches to obtain a closed formula for the quantity in Theorem~\ref{thm:nu},
the expressions we obtained with such approaches are difficult to estimate as $q \to +\infty$ (and we will need these 
asymptotic estimates in Section~\ref{sec:asy}).

To simplify the study of  $\nu_q(N,k,\ell)$, 
throughout this section we 
fix a prime power~$q$, an integer $N \ge 3$  and a vector space $X$ having dimension $N$ over $\F_q$. The particular choice of~$X$ is irrelevant. 
We start by introducing the following simple concepts.

\begin{definition}
\begin{enumerate}[label={(\arabic*)}]
\item 
A \textbf{functional} on 
$X$ is a linear function $f:X \to \F_q$.
The space of functionals on $X$ is denoted by $X^*$.
\item Let
$r \ge 1$ be an integer. The \textbf{kernel}
of an $r$-tuple $F=(f_1,...,f_r) \in (X^*)^r$
is the linear space $\ker(F):=\ker(f_1) \cap \cdots \cap \ker(f_r)$.
\item 
Let $S \subseteq X$ be a set with $0 \in S$ and let $r \ge 1$ be an integer.
We say that $F \in (X^*)^r$ \textbf{distinguishes} $S$ if
$\ker(F)$ distinguishes~$S$. Similarly, we say that~$F$ \textbf{intersects} $S$ if $\ker(F)$ intersects
$S$. In the latter case we write 
$F \triangleright S$. Finally, we let
\begin{align*}
    \tau_q(r,S) := |\{F \in (X^*)^r \mid F \text{ distinguishes } S\}|.
\end{align*}
\end{enumerate}
\end{definition}

A celebrated theorem by Crapo and Rota~\cite[Chapter 16]{crapo1970foundations} expresses $\tau_q(r,S)$ in terms of the combinatorics of the set $S$. More precisely, it shows that 
$\tau_q(r,S)$ is obtained by evaluating the characteristic polynomial of the geometric lattice generated by $S$ at $q^s$.

\begin{definition} \label{def:L}
Let $S \subseteq X$ be a subset with $0 \in S$. We denote by $\mL(S)$ the geometric lattice whose elements are the subspaces of $X$ having a basis made of elements of $S$, ordered by inclusion. 
We also let $\mu_{S}$ and $\rk(S)$ denote its M\"obius function and rank, respectively (note that $\rk(S)$ is
simply the dimension of the 
space generated by the elements of $S$).
The \textbf{characteristic polynomial} of $S$ is
$$\chi(S,\lambda) := \sum_{W \in \mL(S)}
\mu_S(W) \, \lambda^{\rk(S)-\dim(W)} \in \Z[\lambda].$$
\end{definition}

For some sets $S$, the characteristic polynomial $\chi(S,\lambda)$ can be explicitly computed, although this is a very difficult task in general. 

\begin{example}\label{ex:lin}
For
a $k$-space $A \le X$ we have
$\chi(A,\lambda) = \prod_{i=0}^{k-1}(\lambda-q^i).
$ This formula is well-known and follows, for example, from~\cite[Section 3]{stanley1971modular}.
\end{example}

We can now state the result of Crapo and Rota.

\begin{theorem}[\text{see \cite[Chapter 16]{crapo1970foundations}}] \label{thm:CR}
Let 
$S \subseteq X$ be a subset with $0 \in X$ and let $r \ge 1$ be an integer. We have
$$\tau_q(r,S) = q^{r(N-r)} \, \chi(S,q^r).$$
In particular, for all $k$-spaces $A \le X$ we have
$$\tau_q(r,A) = q^{r(N-k)} \prod_{i=0}^{k-1} (q^r-q^i).$$
\end{theorem}

Counting functionals that distinguish a set of vectors is equivalent to counting spaces that distinguish the same set. For some parameters, the mentioned relation between functionals and spaces is particularly simple, as the next lemma illustrates.

\begin{lemma} \label{lem:functionals}
Let
$S \subseteq X$ be a subset with $0 \in S$. Fix any integer $k$ with the property that 
$k \ge \max\{\dim(W) \mid W \le X, \, W \mbox{ distinguishes } S\}$. The number of $k$-spaces $W \le X$ distinguishing~$S$ is 
$$\frac{\tau_q(N-k,S)}{\prod_{i=0}^{N-k-1} (q^{N-k}-q^i)}.$$
\end{lemma}

\begin{proof}
Define the sets
\begin{align*}
    \mathbf{A} &:= \{F \in (X^*)^{N-k} \mid F \text{ distinguishes } S\}, \\
    \mathbf{B} &:= \{W \le X \mid \dim(W) = k, \, W \text{ distinguishes } S\}.
\end{align*}
Let $\varphi: \mathbf{A} \rightarrow \mathbf{B}$ be the map defined by $\varphi: F \mapsto \ker(F)$ for all $F \in \mathbf{A}$. We claim that $\varphi$ is well-defined. To see this, note that $\dim(\ker(F)) \ge N-N+k=k$. Since $k \ge \max\{\dim(W) \mid W \le X, \, W \mbox{ distinguishes } S\}$ by assumption, it must hold that $\dim(\ker(F))=k$. This shows that $\varphi$ is indeed well-defined.
As a next step, we compute the size of the fiber of an arbitrary $W \in \mathbf{B}$ as follows:
\begin{align*}
    |\varphi^{-1}(W)| = |\{F \in (X^*)^{N-k} \mid \ker(F)=W\}|
    =  \prod_{i=0}^{N-k-1}(q^{N-k}-q^i),
\end{align*}
where the latter equality is not difficult to see and left to the reader.
Therefore,
\begin{equation*}
\tau_q(N-k,S) = 
\sum_{ W \in \mathbf{B}}|\varphi^{-1}(W)|
= |\mathbf{B}| \prod_{i=0}^{N-k-1}(q^{N-k}-q^i).
\qedhere
\end{equation*}
\end{proof}

The final step towards
a combinatorial interpretation for $\nu_q(N,k,\ell)$ is the following formula relating tuples of functionals distinguishing linear spaces.
The proof technique combines the aforementioned result by Crapo and Rota 
(Theorem~\ref{thm:CR} above) with Stanley's Modular Factorization Theorem for geometric lattices~\cite{stanley1971modular}.

%%%%%%%%%%%%%%%%%%%%%%%%%%%%%%%%%%
\begin{lemma} \label{lem:prodfunc}
Let $A,B \le X$ be subspaces.
For all $r \ge 1$ we have
$$\tau_q(r,A \cup B) \, \tau_q(r,A \cap B)= \tau_q(r,A) \, \tau_q(r,B).$$
\end{lemma}
%%%%%%%%%%%%%%%%%%%%%%%%%%%%%%%%%%

\begin{proof}
Consider the geometric lattice $\mL(A \cup B)$; see Definition~\ref{def:L}. The rank of $A \cup B$ is $\rk(A\cup B)=\dim(A+B)$. It is easy to see that $A $ is a modular element of $\mL(A \cup B)$ and thus we can use Stanley's Modular Factorization Theorem~\cite[Theorem 2]{stanley1971modular} 
as follows:
\begin{align} 
    \chi (A \cup B, \lambda) &= \chi(A, \lambda)\, \sum_{\substack{W \in \mL(A \cup B)\\ W \cap A =\{0\}}}\mu_{A \cup B}(W)\, \lambda^{\rk(A \cup B)-\dim(A)-\dim(W)} \nonumber \\
    &=\label{pf:prodfunc1} \chi(A, \lambda) \, \sum_{\substack{W \in \mL(B)\\ W \cap A =\{0\}}}\mu_{B}(W)\, \lambda^{\dim(B)-\dim(A \cap B)-\dim(W)}.
\end{align}
We now apply again Stanley's Modular Factorization Theorem to the lattice $\mL(B)$ and the modular element $A \cap B$ of $\mL(B)$, obtaining
\begin{align} \label{pf:prodfunc2}
    \chi (B, \lambda) &= \chi(A \cap B, \lambda)\, \sum_{\substack{W \in \mL(B)\\ W \cap A =\{0\} }}\mu_{B}(W)\, \lambda^{\dim(B)-\dim(A \cap B)-\dim(W)}.
\end{align}
Using~\eqref{pf:prodfunc1} and \eqref{pf:prodfunc2} together we get 
$\chi (A \cup B, \lambda)\, \chi(A \cap B, \lambda) = \chi (A, \lambda) \, \chi (B, \lambda)$.
Finally, the statement of the lemma can easily be derived by combining the latter identity with Theorem~\ref{thm:CR}.
\end{proof}

We are now ready to establish the main result of this section, providing a combinatorial interpretation for $\nu_q(N,k,\ell)$.

\begin{proof}[Proof of Theorem~\ref{thm:nu}]
Fix arbitrary $(N-k)$-spaces $A,B \le X$ that intersect in dimension $\ell$.
The largest dimension of a subspace $W \le X$ that distinguishes $A \cup B$ is at most $k$. Therefore by Lemma~\ref{lem:functionals} we have 
\begin{multline} \label{pf:nu}
    |\{W \le X \mid \dim(W)=k, \, 
W \triangleright A  \mbox{ and } W \triangleright B
\}| = \\
 \qbin{N}{k}{q} - \frac{\tau_q(N-k,A)}{\prod_{i=0}^{N-k-1} (q^{N-k}-q^i)} - \frac{\tau_q(N-k,B)}{\prod_{i=0}^{N-k-1} (q^{N-k}-q^i)} + \frac{\tau_q(N-k,A \cup B)}{\prod_{i=0}^{N-k-1} (q^{N-k}-q^i)}.
\end{multline}
Using Lemma~\ref{lem:prodfunc} we can rewrite the last term of this expression as
\begin{align*}
    \frac{\tau_q(N-k,A \cup B)}{\prod_{i=0}^{N-k-1} (q^{N-k}-q^i)} = \frac{\tau_q(N-k,A)\tau_q(N-k,B)}{\tau_q(N-k,A \cap B)\prod_{i=0}^{N-k-1}(q^{N-k}-q^i)}.
\end{align*}
Finally, by the second part of Theorem~\ref{thm:CR} we have that
\begin{multline*}
    |\{W \le X \mid \dim(W)=k, \, 
W \triangleright A  \mbox{ and } W \triangleright B
\}| = \\
 \qbin{N}{k}{q}  - \frac{2q^{k(N-k)}\prod_{i=0}^{N-k-1} (q^{N-k}-q^i)}{\prod_{i=0}^{N-k-1} (q^{N-k}-q^i)} +  \frac{q^{2k(N-k)}\prod_{i=0}^{N-k-1} (q^{N-k}-q^i)^2}{q^{(N-\ell)(N-k)}\prod_{i=0}^{\ell-1}(q^{N-k}-q^i)\prod_{i=0}^{N-k-1} (q^{N-k}-q^i)},
\end{multline*}
which simplifies to $\nu_q(N,k,\ell)$.
Note moreover that this expression does not depend on the choice of $A$ and $B$, concluding the proof.
\end{proof}

\bigskip

\section{Upper and Lower Bounds} \label{sec:bounds}

In this section we present some of the main results of this paper, providing an answer to Problems~\ref{pb:a} and~\ref{pb:c}. The approach we take
is based on the study of isolated vertices in bipartite graphs. Throughout the paper we use 
the following definition of bipartite graph and isolated vertex.

\begin{definition}
A (\textbf{directed}) \textbf{bipartite graph} is a 3-tuple $\mB=(\mV,\mW,\mE)$, where $\mV$, $\mW$ are finite non-empty sets and $\mE \subseteq \mV \times \mW$. The elements of $\mV \cup \mW$ are called \textbf{vertices}. We say that a vertex~$W \in \mW$ is
\textbf{isolated} if there is no $V \in \mV$ with $(V,W) \in \mE$.

Finally, a bipartite graph $\mB=(\mV,\mW,\mE)$ is \textbf{left-regular} of \textbf{degree} $\partial$ if for all $V \in \mV$ we have
$\partial=|\{W \in \mW \mid (V,W) \in \mE\}|$.
\end{definition}

We start with a very simple upper bound for the number of non-isolated vertices in a left-regular bipartite graph.

\begin{lemma} \label{lem:lowerbound}
Let $\mB=(\mV,\mW,\mE)$ be a bipartite and left-regular graph of degree $\partial>0$.
Let $\mF \subseteq \mW$ be the collection of non-isolated vertices of $\mW$.
We have
$$|\mF| \le |\mV| \, \partial.$$
\end{lemma}
\begin{proof}
We count the elements in the set $\mathbf{A}=\{(V,W) \in \mE \mid V \in \mV, \, W \in \mF\}$ in two ways, obtaining
$$|\mV| \, \partial = |\mathbf{A}| = \sum_{W \in \mF} |\{V \in \mV \mid (V,W) \in \mE\}| \ge |\mF|.$$
The latter inequality follows from the fact that, by assumption,
no vertex in $\mF$ is isolated.
\end{proof}

The next step is to derive a lower bound for the number of non-isolated vertices in a bipartite graph.
We concentrate on a class of such graphs that exhibit strong regularity properties with respect to certain maps defined on their left-vertices.
More precisely, we will use the following concepts.

\begin{definition} \label{def:assoc}
Let $\mV$ be a finite non-empty set and let $r \ge 0$ be an integer. An \textbf{association} on $\mV$ of \textbf{magnitude} $r$ is a function 
$\alpha: \mV \times \mV \to \{0,...,r\}$ that satisfies the following:
\begin{enumerate}[label={(\arabic*)}]%\setlength\itemsep{0cm}
\item $\alpha(V,V)=r$ for all $V \in \mV$;
    \item $\alpha(V,V')=\alpha(V',V)$ for all $V,V' \in \mV$.
    \end{enumerate}
\end{definition}

\begin{definition} \label{def:areg}
Let $\mB=(\mV,\mW,\mE)$ be a finite bipartite graph and let $\alpha$ be an association on~$\mV$ of magnitude $r$.  We say that $\mB$ is \textbf{$\alpha$-regular} if for all  $(V,V') \in \mV \times \mV$ the number of vertices $W \in \mW$ with $(V,W) \in \mE$ and 
$(V',W) \in \mE$ only depends on $\alpha(V,V')$. We denote this number by~$\mW_\ell(\alpha)$, where $\ell=\alpha(V,V') \in \{0,...,r\}$, i.e. we have
\begin{align*}
    \mW_\ell(\alpha) = |\{W \in \mW \mid (V,W) \in \mE, (V',W) \in \mE\}|,
\end{align*}
for any pair $(V,V') \in \mV \times \mV$ such that $\alpha(V,V')=\ell$.
\end{definition}

Note that an $\alpha$-regular bipartite graph $\mB$ as in Definition~\ref{def:areg} is necessarily left-regular of degree $\partial=\mW_r(\alpha)$.
The following lemma gives a lower bound for the number of non-isolated vertices.

\begin{lemma} \label{lem:upperbound}
Let $\mB=(\mV,\mW,\mE)$ be a finite bipartite $\alpha$-regular graph, where $\alpha$ is an association on~$\mV$ of magnitude $r$. Let $\mF \subseteq \mW$ be the collection of non-isolated vertices of $\mW$. If
$\mW_r(\alpha)>0$, then
$$|\mF| \ge  \frac{\mW_r(\alpha)^2 \, |\mV|^2}{\sum_{\ell=0}^r  \mW_\ell(\alpha) \, |\alpha^{-1}(\ell)|}.$$
\end{lemma}

\begin{proof}
Define the set $\mathbf{A}=\{(V,V',W) \in \mV^2 \times \mW \mid (V,W) \in \mE, \, (V',W) \in \mE\}$.
Since all vertices in $\mW \setminus \mF$ are isolated, we have
\begin{equation}\label{ntt1}
    |\mF| \cdot |\mathbf{A}| = |\mF| \, \sum_{W \in \mF} |\{V \in \mV \mid (V,W) \in \mE\}|^2 \ge 
\left( \sum_{W \in \mF} |\{V \in \mV \mid (V,W) \in \mE\}| \right)^2,
\end{equation}
where the latter bound follows from 
the Cauchy-Schwarz Inequality.
As in the proof of Lemma~\ref{lem:lowerbound}, we have
\begin{equation}\label{ntt2}
   \sum_{W \in \mF} |\{V \in \mV \mid (V,W) \in \mE\}| = \sum_{V \in \mV} |\{W \in \mW \mid (V,W) \in \mE\}|
   = \mW_r(\alpha) \, |\mV|.
\end{equation}
Therefore combining \eqref{ntt1} with \eqref{ntt2} we obtain
\begin{equation} \label{ntt3}
    |\mF| \cdot |\mathbf{A}| \ge \mW_r(\alpha)^2 \, |\mV|^2.
\end{equation}
Observe moreover that, by the definition of an association,
\begin{align} \label{ntt4}
    |\mathbf{A}| &=  \sum_{\ell=0}^r \; \sum_{\substack{(V,V') \in \mV^2 \\ \alpha(V,V')=\ell}} |\{W \in \mW \mid (V,W) \in \mE, \, (V',W) \in \mE\}| \nonumber \\
    &= \sum_{\ell=0}^r \mW_\ell(\alpha) \cdot 
    |\{(V,V') \in \mV^2 \mid \alpha(V,V')=\ell\}| \nonumber \\
    &= \sum_{\ell=0}^r  \mW_\ell(\alpha) \, |\alpha^{-1}(\ell)|.
\end{align}
Since $\mW_r(\alpha)>0$, by~\eqref{ntt3}
we have $|\mathbf{A}| \neq 0$. Therefore to conclude the proof it suffices to combine~\eqref{ntt3} with~\eqref{ntt4}.
\end{proof}

We now apply Lemma~\ref{lem:lowerbound} and Lemma~\ref{lem:upperbound}
to derive upper and lower bounds for the number of common complements of a collection of subspaces. This will provide an answer to Problem~\ref{pb:a}, establishing the main result of this section. 

\begin{theorem} \label{thm:boundscc}
Let $X$ be a vector space of finite dimension $N \ge 3$ over $\F_q$ and let $1 \le k \le N-1$ be an integer.
Let $\mA$ be a non-empty collection of subspaces of $X$, all of which have codimension~$k$.
Let $\mF$ be the collection of $k$-spaces $W \le X$ that are not common complements of the spaces in~$\mA$. We have
$$\frac{\nu_q(N,k,N-k)^2 \, |\mA|^2}{\sum_{\ell=0}^{N-k}  \nu_q(N,k,\ell) \cdot |\{(A,A') \in \mA^2 \mid \dim(A\cap A')=\ell\}|}
\le |\mF| \le |\mA| \, \nu_q(N,k,N-k).
$$
In particular, if $|\mA| \ge 2$ and 
$$\ellmax := \max\{\dim(A \cap A') \mid A, A' \in \mA, \, A \neq A'\},$$ then
$$\frac{\nu_q(N,k,N-k)^2\,|\mA|}{\nu_q(N,k,N-k) + (|\mA|-1)\, \nu_q(N,k,\ellmax)} \le |\mF| \le |\mA| \, \nu_q(N,k,N-k).$$
\end{theorem}

\begin{proof}
We apply Lemmas~\ref{lem:lowerbound} and~\ref{lem:upperbound} to the bipartite graph $\mB=(\mA,\mW,\mE)$, 
where~$\mW$ is the collection of $k$-subspaces of $X$ and $(A,W) \in \mE$ if $W$ intersects $A$. We define
an association~$\alpha$ of magnitude $N-k$ on $\mA$ by setting $\alpha(A,A'):=\dim(A \cap A')$ for all $A,\,A' \in \mA$. By Theorem~\ref{thm:nu}, the graph $\mB$ is $\alpha$-regular with 
$\mW_{\ell}(\alpha)=\nu_q(N,k,\ell)$
for all $\ell \in \{0,...,N-k\}$. Note that $\mW_{N-k}(\alpha)=\nu_q(N,k,N-k) > 0$, since every subspace has a complement.
The desired upper and lower bounds on $|\mF|$ now follow directly from   Lemmas~\ref{lem:lowerbound} and~\ref{lem:upperbound}.  

To prove the last part of the statement, observe that the map $\ell \mapsto \nu_q(N,k,\ell)$ is increasing in $\ell$. This can be seen, for example, from Equation~\eqref{pf:nuq1} in the Appendix. Therefore \begin{align*}
\sum_{\ell=0}^{N-k}  \nu_q(N,k,\ell) \, |\alpha^{-1}(\ell)| &\le  \sum_{\ell=0}^{N-k-1}  \nu_q(N,k,\ellmax) \, |\alpha^{-1}(\ell)| + |\mA| \, \nu_q(N,k,N-k) \\
&= \nu_q(N,k,\ellmax) \sum_{\ell=0}^{N-k-1} |\alpha^{-1}(\ell)| +  |\mA| \, \nu_q(N,k,N-k)\\ 
&= \nu_q(N,k,\ellmax) \, |\mA|\, (|\mA|-1) + |\mA|\, \nu_q(N,k,N-k).
\end{align*}
Combining this with Lemma~\ref{lem:upperbound} we obtain
\begin{equation*}
    |\mF| \ge \frac{\nu_q(N,k,N-k)^2\,|\mA|^2 }{|\mA|\, \nu_q(N,k,N-k) + |\mA| \, (|\mA|-1) \,  \nu_q(N,k,\ellmax)}, 
\end{equation*}
which is the lower bound in the second part of the statement.
\end{proof}

Note that the lower bound on $|\mF|$ in Theorem~\ref{thm:boundscc} takes into account the number of spaces to be complemented, but also their ``intersection structure''. More precisely,
it takes into account how many subspace pairs intersect in a given dimension.
The latter information will be crucial when deriving upper bounds on the density function of MRD codes in Section~\ref{sec:rankmetric}; see in particular Theorem~\ref{thm:mrdboundcc}.

\begin{remark}
A lower bound for the number of common complements of a collection of subspaces was  obtained in~\cite[Theorem 5]{tingley1991complements}. Following the notation and the assumptions of Theorem~\ref{thm:boundscc}, the result of~\cite{tingley1991complements} states that if
$|\mA| \le q$, then the number of common complements of the spaces in~$\mA$ is at least $q+1-|\mA|$. For $|\mA|=q$, the lower bound of~\cite{tingley1991complements} is 1, whereas the bound given in Theorem~\ref{thm:boundscc} is negative. Therefore, for this particular case, the result of~\cite{tingley1991complements} is sharper. For $|\mA| < q$ and sufficiently large $q$, it is possible to check that the bound following from Theorem~\ref{thm:boundscc} is at least as good as the one in~\cite{tingley1991complements}.
% It is possible to check that for $|\mA| < q$ and large enough $q$, the bound given in~Theorem~\ref{thm:boundscc} performs better than the one of~\cite{tingley1991complements}.

% Straightforward computations show that the lower bound on the number of common complements following from Theorem~\ref{thm:boundscc} is at least as sharp as the one of~\cite{tingley1991complements} if and only if, for fixed $N$ and $k$, the function $h(i)=\qbin{N}{k}{q}-i(\qbin{N}{k}{q}-q^{k(N-k)})-2$ is non-negative for all $i \in \{1,\dots,q-1\}$. As this is a linear function which is decreasing in $i$ (which will be $|\mA|$), it suffices to show that $h(q-1) \ge 0$. As we are interested in the asymptotic behavior as $q \to +\infty$ we computed the asymptotic estimate of $h(q-1)$,
% \begin{align*}
%     h(q-1) \sim q^{k(N-k)-1}
% \end{align*}
% which for $k \le N-1$ is always $\gg q$ and thus in particular positive. Note that we used the fact that $\qbin{N}{k}{q}-q^{k(N-k)} \sim q^{k(N-k)-1}$ as $q \to +\infty$. This is not obvious, but it follows from the fact that we have 
% $$\qbin{N}{k}{q} = \qbin{N-1}{k-1}{q}+q^k\qbin{N-1}{k}{q}.$$
% Now if $k\in \{1,N-1\}$ then we trivially have $\qbin{N}{k}{q} = \frac{q^N-1}{q^k-1} = q^{N-1}+q^{N-2}+\dots$. If $k \ge 2$ and $k \le N-2$ then $\qbin{N}{k}{q}=q^{(k-1)(N-k)}+q^{(k-1)(N-k)-1}+\operatorname{lot}+q^k(q^{k(N-k-1)}+q^{k(N-k-1)-1}+\operatorname{lot})$. Since $N \ge k+2$ we have $(k-1)(N-k) <k+k(N-k-1)-1$ so the leading terms are $q^k(q^{k(N-k-1)}+q^{k(N-k-1)-1})$.
\end{remark}

We conclude this section with an upper and lower bound for the number of spaces of a given dimension that distinguish a given cone. This provides an answer to Problem~\ref{pb:c}.

\begin{theorem} \label{thm:boundsv}
Let $X$ be a vector space of finite dimension $N \ge 3$ over $\F_q$ and let $K \subseteq X$ be a cone with $|K| \ge q$. Let
$1 \le k \le N-1$ be an integer.
Denote by $\mF$ 
 the collection of $k$-subspaces $W \le X$ 
that intersect $K$. We have
$$\displaystyle\frac{\displaystyle\frac{|K|-1}{q-1}\, \qbin{N-1}{k-1}{q}}{1+\displaystyle\left(\frac{|K|-1}{q-1}-1\right)\displaystyle\left(\frac{q^{k-1}-1}{q^{N-1}-1}\right)} \le |\mF| \le \displaystyle\frac{|K|-1}{q-1} \, \qbin{N-1}{k-1}{q}.$$
\end{theorem}

\begin{proof}
This time we apply Lemmas~\ref{lem:lowerbound} and~\ref{lem:upperbound} to the bipartite graph $\mB=(\mV,\mW,\mE)$ defined as follows: $\mV$ is the collection of $1$-subspaces generated by the nonzero elements of $K$, $\mW$ is the collection of $k$-subspaces of $X$, and 
$(L,W) \in \mE$ if $L \le W$. We further define an association~$\alpha$ on $\mV$ by setting $\alpha(V,V'):=\dim(V \cap V')$ for all $V,\,V' \in \mV$. It is easy to see that $\alpha$ has magnitude 1 and that the graph~$\mB$ is $\alpha$-regular. Moreover, 
$$\mW_0=\qbin{N-2}{k-2}{q}, \quad \mW_1=\qbin{N-1}{k-1}{q},
\quad |\alpha^{-1}(0)|=|\mV|\, (|\mV|-1), \quad |\alpha^{-1}(1)|=|\mV|.$$
The upper bound on $|\mF|$ is
an immediate consequence of Lemma~\ref{lem:lowerbound} and the fact that~$\mB$ is left-regular of degree~$\mW_1$, as observed right after Definition~\ref{def:areg}. Furthermore, by applying Lemma~\ref{lem:upperbound} we get
$$|\mF| \ge  \displaystyle\frac{|\mV|\, \qbin{N-1}{k-1}{q}^2}{\qbin{N-1}{k-1}{q}+\left(|\mV|-1\right)\, \qbin{N-2}{k-2}{q}}.$$
The lower bound in the statement can easily be obtained from this inequality using the fact that $|\mV| =(|K|-1)/(q-1)$,
along with the definition of the $q$-binomial coefficient and some straightforward computations.
\end{proof}

\begin{remark}\label{rmk:sperner}
While drafting this paper, we found that the lower bound on $|\mF|$ in Theorem~\ref{thm:boundsv} can also be derived from 
a known result in {Sperner theory}.
More precisely, lengthy computations show that
the lower bound of Theorem~\ref{thm:boundsv} coincides with that
of~\cite[Lemma~12]{bey2004polynomial} for $\ell=1$ (and the same value of~$k$). This result is used in~\cite{bey2004polynomial} towards the derivation of a \textit{polynomial LYM inequality} for the linear lattice and is more general than our Theorem~\ref{thm:boundsv}.
Although both proofs partially rely on the Cauchy-Schwarz Inequality,
our argument has a more 
``enumerative'' flavor thanks 
to the concept of an association. This  allows us to avoid the eigenvalue machinery in the proof of~\cite{bey2004polynomial}.
In this paper, the best bounds are 
 obtained by 
applying Theorem~\ref{thm:boundscc} (rather than Theorem~\ref{thm:boundsv}),
which is instead not related to the problems studied in~\cite{bey2004polynomial}.
\end{remark}

\begin{remark} \label{rmk:a_vs_d2}
We continue the discussion 
started in Remark~\ref{rmk:a_vs_d1}.
Even though the union of linear subspaces is a cone, Theorem~\ref{thm:boundscc} 
and Theorem~\ref{thm:boundsv} have different applicability.
 Theorem~\ref{thm:boundscc} can be used when information about the intersection structure of the subspaces to be complemented (along with their number) is known,
 without requiring any particular knowledge about the cardinality of their union. Vice versa, Theorem~\ref{thm:boundsv} can be used to give an answer to Problem~\ref{pb:a} when 
 the size of $\bigcup \mA$ is known.

  The following example illustrates two situations in which the information needed to compute the lower bounds in  Theorems~\ref{thm:boundscc} 
and~\ref{thm:boundsv} is completely available, showing that Theorem~\ref{thm:boundscc} provides a sharper bound in both scenarios. This will be the case also when estimating the density function of MRD codes in Sections~\ref{sec:rankmetric} and~\ref{sec:m}.
\end{remark}

\begin{example}\label{exe:a_vs_d} 
\begin{enumerate}
    \item Let $X$ be a vector space of dimension $5$ over $\F_2$. Select a subspace $X' \le X$ of dimension $4$ and let $\mA$ be a 2-spread of $X'$; see~\cite[Chapter~4]{hirschfeld1998projective}. Then any two (distinct) elements of $\mA$ intersect in $\{0\}$ and $|\mA|=(2^4-1)/(2^2-1)=5$. Denote by $\mF$ the family of $3$-subspaces of $X$ that intersect at least one element of $\mA$. The lower bound of Theorem~\ref{thm:boundscc} reads $|\mF| \ge 141$ and the one of Theorem~\ref{thm:boundsv} reads instead $|\mF| \ge 139$. Since $\bigcup \mA$ is a 4-dimensional subspace of $X$, all 155 subspaces of~$X$ of dimension 3 intersect some element of $\mA$. 
    
    \item Let $X$ be a vector space of dimension~$5$ over $\F_2$. Select a subspace $X' \le X$ of dimension~$3$ and let $\mA$ be the collection of 2-subspaces of $X'$.
    We have  $|\mA|=\qbin{3}{2}{2}=7$. Denote by $\mF$ the family of $3$-subspaces of $X$ that intersect at least one element of $\mA$. The lower bound of Theorem~\ref{thm:boundscc} reads $|\mF| \ge 131$, while that of Theorem~\ref{thm:boundsv} reads $|\mF| \ge 112$ and is therefore coarser. Again, all 155 subspaces of $X$ of dimension 3 intersect at least one element of $\mA$.
\end{enumerate}
\end{example}

\bigskip

\section{Asymptotic Results} \label{sec:asy}

This section is entirely devoted to
the asymptotic versions of Theorems~\ref{thm:boundscc} and~\ref{thm:boundsv}.
These will be stated in the following
language.

\begin{notation}
We use the Bachmann-Landau notation (``Big O'', ``Little O'', and~``$\sim$'') to describe the asymptotic growth of real-valued functions defined on an infinite set of natural numbers; see e.g.~\cite{de1981asymptotic}.  We also denote by $Q$ the set of prime powers and omit ``$q \in Q$'' when writing $q \to +\infty$.
\end{notation}

In the remainder of the paper we will repeatedly need the asymptotic estimate 
for the $q$-binomial coefficient as $q$ grows, i.e.,
\begin{equation} \label{eq:qbin}
\qbin{a}{b}{q} \sim q^{b(a-b)} \quad \mbox{as $q \to +\infty$}
\end{equation}
for all integers $a \ge b \ge 0$.
In the sequel we will apply this well-known fact without explicitly
referring to it.

For convenience of exposition and to simplify arguments in the sequel, we start by establishing the asymptotic
version of Theorem~\ref{thm:boundsv}.

\begin{theorem} \label{thm:asyv}
Let $(X_q)_{q \in Q}$ be a sequence of vector spaces, all of which have the same dimension $N \ge 3$ over $\F_q$. Let $(K_q)_{q \in Q}$ be a sequence of cones with $K_q \subseteq X_q$ and $|K_q| \ge q$ for all $q \in Q$, and let $1 \le k \le N-1$ be an integer. For $q \in Q$, denote by $\mF_q$ and $\mF_q'$
the collections of $k$-subspaces $W_q \le X_q$ intersecting and distinguishing 
$K_q$, respectively. The following hold.
\begin{enumerate}[label={(\arabic*)}]%\setlength\itemsep{0cm}
    \item \label{thm:asyv1} We have  $$\frac{|\mF_q|}{\qbin{N}{k}{q}} \in O\left(\frac{|K_q|}{q^{N-k+1}}\right)  \quad \mbox{as $q \to +\infty$.}$$ 
    In particular, if $|K_q| \in  o(q^{N-k+1})$ as $q \to +\infty$, then
    $$\lim_{q \to +\infty} \frac{|\mF_q|}{\qbin{N}{k}{q}} =0.$$
    
    \item \label{thm:asyv2} We have
    $$\frac{|\mF'_q|}{\qbin{N}{k}{q}} \in O\left(\frac{q^{N-k+1}}{|K_q|}\right)  \quad \mbox{as $q \to +\infty$.}$$
In particular, if $ q^{N-k+1} \in o(|K_q|)$ as $q \to +\infty$, then
    $$\lim_{q \to +\infty} \frac{|\mF'_q|}{\qbin{N}{k}{q}} =0.$$
    
\end{enumerate}
\end{theorem}
\begin{proof}
In the sequel, all asymptotic estimates are for $q \to + \infty$. By the upper bound in Theorem~\ref{thm:boundsv} we have
\begin{equation*}
\frac{|\mF_q|}{\qbin{N}{k}{q}} \le 
\frac{\displaystyle\frac{|K_q|-1}{q-1} \, \qbin{N-1}{k-1}{q}}{\qbin{N}{k}{q}},
\end{equation*}
which establishes the first part of the statement by taking the limit.
For the second part, observe first that $|\mF_q| + |\mF'_q|=\qbin{N}{k}{q}$.
Therefore by the lower bound for $|\mF_q|$ in Theorem~\ref{thm:boundsv} we have
\begin{align} \label{pf:asyv0}
\frac{|\mF'_q|}{\qbin{N}{k}{q}} \le 1-
\frac{\displaystyle\frac{|K_q|-1}{q-1}\,\qbin{N-1}{k-1}{q}}{\qbin{N}{k}{q} \left(1+\displaystyle\left(\frac{|K_q|-1}{q-1} -1\right)\displaystyle\left(\frac{q^{k-1}-1}{q^{N-1}-1}\right)\right)}.
\end{align}
The latter inequality can be rewritten as
\begin{equation} \label{pf:asyv1}
\frac{|\mF'_q|}{\qbin{N}{k}{q}} \le 
\frac{1 -\displaystyle\frac{q^{k-1}-1}{q^{N-1}-1}- \displaystyle\frac{|K_q|-1}{q-1}\left(\frac{\qbin{N-1}{k-1}{q}}{\qbin{N}{k}{q}} -\displaystyle\frac{q^{k-1}-1}{q^{N-1}-1}\right)}{1+\displaystyle\left(\frac{|K_q|-1}{q-1} -1\right)\displaystyle\left(\frac{q^{k-1}-1}{q^{N-1}-1}\right)}.
\end{equation} 
Note that the quantity
    \begin{align*}
        \frac{\qbin{N-1}{k-1}{q}}{\qbin{N}{k}{q}} -\displaystyle\frac{q^{k-1}-1}{q^{N-1}-1} =  \displaystyle\frac{q^{k}-1}{q^{N}-1} - \displaystyle\frac{q^{k-1}-1}{q^{N-1}-1}
    \end{align*}
    is positive for any $q$. Thus from~\eqref{pf:asyv1} we obtain 
\begin{align*}
\frac{|\mF'_q|}{\qbin{N}{k}{q}} &\le \frac{1 -\displaystyle\frac{q^{k-1}-1}{q^{N-1}-1}}{1+\displaystyle\left(\frac{|K_q|-1}{q-1} -1\right)\displaystyle\left(\frac{q^{k-1}-1}{q^{N-1}-1}\right)} \le 
\frac{1}{\displaystyle\left(\frac{|K_q|-1}{q-1} -1\right)\displaystyle\left(\frac{q^{k-1}-1}{q^{N-1}-1}\right)},
\end{align*}
from which the second part of the statement follows easily by taking the limit.
\end{proof}

\begin{remark} \label{rmk:mayornot}
Theorem~\ref{thm:asyv} does not predict any asymptotic behavior in the case where 
$|K_q| \sim \gamma \, q^{N-k+1}$ as $q \to +\infty$  for some constant $\gamma \in \R_{>0}$.
The following Proposition~\ref{prop:onehalf} shows that for any such $\gamma$ the family $\smash{(\mF_q')_{q \in Q}}$ is not dense. 
Right after Proposition~\ref{prop:onehalf} we will include an example showing that,
when $\smash{|K_q| \sim q^{N-k+1}}$ as $q \to +\infty$,  $(\mF_q')_{q \in Q}$ may or may not be sparse.
\end{remark}

The following result is easy to obtain by computing the asymptotics in~\eqref{pf:asyv0}. The details of the proof are omitted.

\begin{proposition} \label{prop:onehalf} 
Let $(X_q)_{q \in Q}$ be a sequence of vector spaces, all of which have the same dimension $N \ge 3$ over $\F_q$. Let $(K_q)_{q \in Q}$ be a sequence of cones with $K_q \sim \gamma \, q^{N-k+1}$ for $q \to +\infty$, where $\gamma \in \R_{>0}$ is a constant.
Let $1 \le k \le N-1$ be an integer and for $q \in Q$ denote by $\mF'_q$
the collection of $k$-subspaces $W_q \le X_q$ distinguishing 
$K_q$. We have 
$$\limsup_{q \to +\infty} \, \frac{|\mF'_q|}{\qbin{N}{k}{q}} \le \frac{1}{\gamma+1} < 1.$$
\end{proposition}

We now provide the examples mentioned in Remark~\ref{rmk:mayornot}, illustrating two possible behaviors in the case $|K_q| \sim q^{N-k+1}$ as $q \to +\infty$.

\begin{example}
\begin{enumerate}[label={(\arabic*)}]%\setlength\itemsep{0cm}
\item Let $(X_q)_{q \in Q}$ be a sequence of linear spaces, all of which have the same dimension $N \ge 3$ over $\F_q$. Let $1 \le k \le N-1$ be an integer and fix a sequence $(K_q)_{q \in Q}$ of $(N-k+1)$-spaces with $K_q \le X_q$ for all $q \in Q$. By dimension considerations, for all $q \in Q$, there are no $k$-spaces in $X_q$ avoiding the cone $K_q$. In particular, the $k$-spaces avoiding $K_q$ are trivially sparse.
\item By definition, the $m$-dimensional rank-metric codes in $\F_q^{2 \times m}$ of minimum distance $2$ are exactly the $m$-dimensional spaces distinguishing the cone of matrices of rank strictly smaller than 2. There are $\smash{\bbq{2 \times m, 1} -1\sim q^{m+1}}$ such matrices as $q \to +\infty$; see the estimate in~\eqref{asball}. We also have $\smash{\delta_q(2 \times m, m, 2) \sim \sum_{i=0}^m (-1)^i/i! >0}$ as
$q \to +\infty$, as shown in~\cite[Corollary~VII.5]{antrobus2019maximal}. In particular, the $m$-dimensional subspaces of $\F_q^{2 \times m}$ distinguishing the ball in $\F_q^{2 \times m}$ of radius $1$ are not sparse.
\end{enumerate}
\end{example}

We now turn to the main result of this section (Theorem~\ref{thm:asycc}). As we will see later, of particular interest for the study of MRD codes are families of linear spaces
that, asymptotically, behave like a \textit{partial spread} (we refer the reader to~\cite{beutelspacher1979t} for the notion of partial spread in finite geometry). More precisely, we propose the following concept.

\begin{definition} \label{def:asyspread}
Let $(X_q)_{q \in Q}$ be a sequence of vector spaces of the same dimension $N \ge 3$ over $\F_q$. Let $(\mA_q)_{q \in Q}$ be a sequence of collections of subspaces $A_q \le X_q$, all of which have the same dimension $k$. We say that $(\mA_q)_{q \in Q}$ is an \textbf{asymptotic partial spread}
if
$$\left| \bigcup_{A_q \in \mA_q}  A_q\right| \sim |\mA_q| \, q^k   \quad \mbox{as $q \to +\infty$},$$
i.e., if the cardinality of the union $\bigcup_{A_q \in \mA_q} A_q$ has the largest possible asymptotics for the given parameters.
\end{definition}

We are now ready to state the asymptotic version of Theorem~\ref{thm:boundscc}.
The result gives asymptotic estimates for the proportion of common complements of a collection of subspaces. Notice that Part~\ref{pspread} of the next theorem will play a central role in establishing Theorem~\ref{sparseness}, which is one of the main results of this paper.

\begin{theorem} \label{thm:asycc}
Let $(X_q)_{q \in Q}$ be a sequence of vector spaces, all of which have the same dimension $N \ge 3$ over $\F_q$.
Let $1 \le k \le N-1$ be an integer and let $(\mA_q)_{q \in Q}$ be a sequence of non-empty collections of linear spaces,
all of which have codimension $k$,
with $A_q \le X_q$ for all $q \in Q$ and $A_q \in \mA_q$.
For $q \in Q$, denote by~$\mF_q$ and~$\mF'_q$
the collections of $k$-subspaces $W_q \le X_q$ that intersect some $A_q \in \mA_q$ and that distinguish every~$A_q \in \mA_q$, respectively. Then the following hold. 
\begin{enumerate}[label={(\arabic*)}]
    \item We have $$    \frac{|\mF_q|}{\qbin{N}{k}{q}} \in O \left( \frac{|\mA_q|}{q}  \right) \quad \mbox{as $q \to +\infty$.}$$
    In particular, if $|\mA_q| \in o(q)$ as $q \to +\infty$, then
    $$\lim_{q \to +\infty} \frac{|\mF_q|}{\qbin{N}{k}{q}} =0.$$

    \item \label{pspread} Suppose that $(\mA_q)_{q \in Q}$ is an asymptotic partial spread.
    Then
    $$    \frac{|\mF'_q|}{\qbin{N}{k}{q}} \in O \left( \frac{q}{|\mA_q|}  \right) \quad \mbox{as $q \to +\infty$.}$$
    In particular, if $q \in o(|\mA_q|)$ as $q \to +\infty$, then
    $$\lim_{q \to +\infty} \frac{|\mF'_q|}{\qbin{N}{k}{q}} =0.$$
    
    \item Suppose that $q \in o(|\mA_q|)$ as $q \to +\infty$ and that there exist $\overline{q} \in Q$ and an integer $\ell$ with 
    $\max\{0,N-2k\} \le \ell < N-k-1$ or $\ell = \max\{0,N-2k\} = N-k-1$ that satisfy the following property: 
    \begin{align} \label{maxx}
        \max\{\dim(A_q \cap A'_q) \mid (A_q, A'_q) \in \mA_q^2, \, A_q \neq A'_q\} \le \ell \quad \mbox{for all $q \in Q$, $q \ge \overline{q}$}.
    \end{align} 
    Then the following hold.
    
    \begin{enumerate}[label={(3\alph*)}]%\setlength\itemsep{0cm}
    \item \label{case3a} If $\ell = \max\{0,N-2k\}$, then
    $$    \frac{|\mF'_q|}{\qbin{N}{k}{q}} \in O \left( \frac{q}{|\mA_q|}  \right) \quad \mbox{as $q \to +\infty$.}$$
    
    \item \label{case3b} If $\ell > \max\{0,N-2k\}$, then
    $$    \frac{|\mF'_q|}{\qbin{N}{k}{q}} \in O \left( \frac{q}{|\mA_q|}  + q^{-N+k+\ell+1} \right) \quad \mbox{as $q \to +\infty$.}$$
    \end{enumerate}
    In either case we have
    $$\lim_{q \to +\infty} \frac{|\mF'_q|}{\qbin{N}{k}{q}} =0.$$
\end{enumerate}
\end{theorem}

Before proceeding with the proof of Theorem~\ref{thm:asycc}, we describe its statement from a more ``qualitative'' viewpoint.

\begin{remark}
Theorem~\ref{thm:asycc} illustrates the general behavior of the common complements 
of the spaces in $(\mA_q)_{q \in Q}$ as the field size grows. With the only exception when $\ell=N-k-1$ (case in which we are not able to predict the behavior),
the decisive property for sparsity/density~is whether or not the integer sequence $(|\mA_q|)_{q \in Q}$ is negligible with respect to the field size $q$ in the asymptotics. It is interesting to observe that Theorem~\ref{thm:asycc} does not extend to the case where, for example,
$|\mA_q| \sim q$ as $q \to +\infty$. We will elaborate on this at the~end of the section; see Example~\ref{counter_cc}.
\end{remark}

In the remainder of the section we establish Theorem~\ref{thm:asycc}. We start with a technical lemma, whose proof can be found in the Appendix.

%%%%%%%%%%%%%%%%%%%%%%%%%%%%%%%%%%

%%%%%%%%%%%%%%%%%%%%%%%%%%%%%%%%%%
\begin{lemma} \label{lem:nuq}
Let $N$, $k$ and $\ell$ be integers as in Notation~\ref{notat:nu}.
The following estimates hold 
 as $q \to + \infty$:
\begin{align*}
    \nu_q(N,k,\ell) &\sim \begin{cases}
    q^{k(N-k)-2} &\textnormal{ if } \max\{0,N-2k\} \le \ell < N-k-1 \\
    &\textnormal{ or } \ell = N-k-1 = \max\{0,N-2k\}, \\
    2q^{k(N-k)-2} &\textnormal{ if } \ell = N-k-1 > \max\{0,N-2k\}, \\
    q^{k(N-k)-1} &\textnormal{ if } \ell = N-k,
    \end{cases}   \\
    & \ \\
    \frac{\nu_q(N,k,N-k)^2}{\qbin{N}{k}{q}} - \nu_q(N,k,\ell) &\sim \begin{cases}
    q^{k(N-k)-N+k-1} &\textnormal{ if } \ell = N-2k, \\
    q^{k(N-k)-k-1} &\textnormal{ if } \ell = 0, \\
    -q^{k(N-k)-N+k+\ell-1} &\textnormal{ if } \ell > \max\{0,N-2k\}.
    \end{cases}
\end{align*}
\end{lemma}

\begin{proof}[Proof of Theorem~\ref{thm:asycc}] 
All asymptotic estimates in this proof are for $q \to + \infty$. We examine the three cases in the statement separately.
\begin{enumerate}[label={(\arabic*)}]

\item 
By Theorem~\ref{thm:boundscc} we have
$$\frac{|\mF_q|}{\qbin{N}{k}{q}}  \le \frac{|\mA_q|}{\qbin{N}{k}{q}} \, \nu_q(N,k,N-k),$$
which, together with Lemma~\ref{lem:nuq}, gives the desired asymptotic estimate.

\item
Consider the cone $$K_q := \bigcup_{A_q \in \mA_q}  A_q.$$
Since $\mA_q$ is an asymptotic partial spread by assumption, we have $K_q \sim |\mA_q|\,q^{N-k}$. The statement now follows from Theorem~\ref{thm:asyv}\ref{thm:asyv2}.

\item
Denote by 
$\ellmax$ the maximum on the LHS of~\eqref{maxx}.
Since 
$|\mF_q| + |\mF'_q|=\qbin{N}{k}{q}$, the lower bound for $|\mF_q|$ in Theorem~\ref{thm:boundscc}  tells us that, for all $q \ge \overline{q}$, 
\begin{align} 
    \frac{|\mF'_q|}{\qbin{N}{k}{q}} &\le  \nonumber 1-\frac{\nu_q(N,k,N-k)^2\,|\mA_q|}{\qbin{N}{k}{q}\left(\nu_q(N,k,N-k) + (|\mA_q|-1)\, \nu_q(N,k,\ellmax)\right)}\\
    &\le \label{pf:asycc0} 1-\frac{\nu_q(N,k,N-k)^2\,|\mA_q|}{\qbin{N}{k}{q}\left(\nu_q(N,k,N-k) + (|\mA_q|-1)\, \nu_q(N,k,\ell)\right)},
\end{align}
where the latter inequality follows from the fact that $\ell \mapsto \nu_q(N,k,\ell)$ is increasing (for this, see again Equation~\eqref{pf:nuq1} in the Appendix). Define the difference
     \begin{equation*} \label{pf:asycc4}
      \Delta_q(N,k,\ell):= \frac{\nu_q(N,k,N-k)^2}{\qbin{N}{k}{q}}-\nu_q(N,k,\ell).  \end{equation*}
We rewrite the inequality in~\eqref{pf:asycc0} as follows:
    \begin{equation} \label{pf:asycc1} \frac{|\mF'_q|}{\qbin{N}{k}{q}} \le \frac{\nu_q(N,k,N-k)-\nu_q(N,k,\ell)-|\mA_q|\,\Delta_q(N,k,\ell)}{\nu_q(N,k,N-k) + (|\mA_q|-1)\, \nu_q(N,k,\ell)}.
    \end{equation}
    Since $q \in o(|A_q|)$ and $\max\{0,N-2k\} \le \ell < N-k-1$ or $\ell = \max\{0,N-2k\} = N-k-1$, 
    using the asymptotic estimates from Lemma~\ref{lem:nuq} we get
    \begin{align} \label{pf:asycc2}
    \nu_q(N,k,N-k)-\nu_q(N,k,\ell) &\sim q^{k(N-k)-1}, \\ \label{pf:asycc3}
    \nu_q(N,k,N-k) + (|\mA_q|-1) \, \nu_q(N,k,\ell) &\sim |\mA_q| \, q^{k(N-k)-2}.
    \end{align}

      If $\ell =\max\{0,N-2k\}$, then Lemma~\ref{lem:nuq}
      tells us that $\Delta_q(N,k,\ell)$ is 
      positive for $q$ sufficiently large. Therefore case~\ref{case3a} follows by combining~\eqref{pf:asycc1}, \eqref{pf:asycc2}, and~\eqref{pf:asycc3}. In particular, $\lim_{q \to +\infty} |\mF'_q|/\qbin{N}{k}{q} =0$.
      
      Now suppose that $\ell > \max\{0,N-2k\}$, which in turn implies $\ell < N-k-1$.
      Using again Lemma~\ref{lem:nuq}
      we obtain
      \begin{equation*}
       \frac{\Delta_q(N,k,\ell)}{q^{k(N-k)-2}} \sim 
       -q^{-N+k+\ell+1}.
      \end{equation*}
      Combining this estimate with~\eqref{pf:asycc1}, \eqref{pf:asycc2}, and~\eqref{pf:asycc3} one establishes case~\ref{case3b}.
Finally, the fact that $\lim_{q \to +\infty} |\mF'_q|/\qbin{N}{k}{q} =0$ follows from $q \in o(|\mA_q|)$ and 
$\ell < N-k-1$, which implies
$-N+k+\ell+1 \le -1$. \qedhere
\end{enumerate}
\end{proof}

We conclude this section with two examples focusing on the case 
$|\mA_q| \sim q$ as $q \to +\infty$, which is not covered by Theorem~\ref{thm:asycc}.
We show that in such case the common complements can be sparse or not.

\begin{example} \label{counter_cc}
\begin{enumerate}[label={(\arabic*)}]%\setlength\itemsep{0cm}
    \item Let $(X_q)_{q \in Q}$ be a sequence of linear spaces, all of which have the same dimension $N \ge 3$ over $\F_q$.
Fix a sequence $(V_q)_{q \in Q}$ of 2-dimensional spaces $V_q \le X_q$. For $q \in Q$, denote by~$\mA_q$ be the set of 1-dimensional subspaces of $V_q$. We then have $|\mA_q| \sim \qbin{2}{1}{q} \sim q$ as $q \to +\infty$. In particular, there are no $(N-1)$-dimensional subspaces of $X_q$ that distinguish $V_q$ and the common complements of the spaces in~$\mA_q$ are sparse.
\item Let $m \ge 2$ be an integer. By Remark~\ref{rem:expl}, the MRD codes of minimum distance~$2$ in $\F_q^{2\times m}$ are the common complements of $\qbin{2}{1}{q} \sim q$ subspaces of $\F_q^{2\times m}$ having dimension $m$, where the estimate is for $q \to +\infty$. Their asymptotic density is
$\delta_q(2 \times m, m, 2) \sim \sum_{i=0}^m (-1)^i/i! >0$ as
$q \to +\infty$; see \cite[Corollary VII.5]{antrobus2019maximal} and the discussion right after our Theorem~\ref{thm:heide}. In particular, they are not sparse.
\end{enumerate}

\end{example}

\bigskip

\section{The Density Function of Rank-Metric Codes} \label{sec:rankmetric}

In this section we apply the theory developed in the previous sections to matrix spaces over a finite field, obtaining upper and lower bounds for the density functions of MRD codes.
By computing the limit as $q \to +\infty$ in these bounds we then 
solve Problem~\ref{pb:b}, stated in the introduction of this paper. In particular,
 we prove that MRD codes in $\mat$ of minimum distance $d$ are sparse
unless $d=1$ or $d=n=2$.

Before presenting the main theorems of this section and their proofs, we briefly survey the current literature 
connected to Problem~\ref{pb:b}.
This will also serve to put our results in the context of previous work.

\begin{notation}
For ease of exposition, throughout this section
we work with fixed integers~$m$,~$n$ and~$d$ with $m \ge n \ge 2$ and $1 \le d \le n$.
\end{notation}

The density limit considered in Problem~\ref{pb:b} has been 
studied in~\cite{byrne2020partition,antrobus2019maximal,gluesing2020sparseness}, showing in particular that
$$\limsup_{q \to +\infty} \, \delta(n \times m, m(n-d+1), d) <1 \mbox{ whenever $d \ge 2$}.$$
This result appears quite surprising when thinking 
of MRD codes as the rank-metric analogues of MDS codes in the Hamming metric, which are classically known to be dense.
It turns out that the approach developed in this paper provides a clear explanation for the divergence in the behavior of these two classes of codes; see Remark~\ref{rem:dive} below. 

The methods used in \cite{byrne2020partition}, \cite{antrobus2019maximal} and
\cite{gluesing2020sparseness} are very different from each other. The approach of
\cite{byrne2020partition} uses a  combinatorial machinery based on families of codes that are \textit{balanced} with respect to a given partition of the ambient space, leading to the following result.

\begin{theorem}[\text{see \cite[Corollary 6.2]{byrne2020partition}}] \label{thm:ravby} If $d \ge 2$, then
\begin{align*}
    \limsup_{q \to +\infty}\, \delta_q(n \times m, m(n-d+1), d) \le 1/2.
\end{align*}
\end{theorem}

A sharper bound is obtained 
in \cite{antrobus2019maximal} using the theory of
spectrum-free matrices, combined with a probability argument.
The result reads as follows.

\begin{theorem}[\text{see \cite[Theorem VII.6]{antrobus2019maximal}}] \label{thm:heide}
We have
\begin{align*}
    \limsup_{q \to +\infty} \, \delta_{q}(n \times m, m(n-d+1), d) \le \left(\sum_{i=0}^{m}\frac{(-1)^i}{i!}\right)^{(d-1)(n-d+1)}.
\end{align*}
\end{theorem}

In \cite{antrobus2019maximal} it is also shown that the bound of
Theorem~\ref{thm:heide} is sharp
whenever $d=n=2$ (and for arbitrary $m \ge 2$).
This means that, in general, $\MRD$ codes are neither sparse, nor dense.

Finally, in \cite{gluesing2020sparseness} the exact density of MRD codes with parameters $m=n=d=3$ is computed, showing that these $3 \times 3$ codes are sparse. The approach of~\cite{gluesing2020sparseness} is based on 
an original argument that connects full-rank square MRD codes with the theory of semifields.

\begin{theorem}[\text{see \cite[Theorem 2.4]{gluesing2020sparseness}}]
We have 
\begin{align*}
    \delta_q(3 \times 3,3,3) = \frac{(q-1)\, (q^3-1)\, (q^3-q)^3\, (q^3-q^2)^2\, (q^3-q^2-q-1)}{3\, (q^7-1)\, (q^9-1)\, (q^9-q)}.
\end{align*}
In particular,
$\lim_{q \to +\infty} \, \delta_q(3 \times 3, 3,3) = 0$.
\end{theorem}

In this paper we approach Problem~\ref{pb:b} from a different viewpoint,
which allows us to obtain
sharper bounds for the density function of MRD codes.
As an application, we conclude that MRD codes are  sparse as $q \to +\infty$, unless $d=1$ or $n=d=2$. We therefore show that the non-sparseness result of~\cite{antrobus2019maximal} for $d=n=2$ is the only non-trivial exception to a general ``sparseness behavior''.

We start with 
an upper bound on the density of MRD codes, which is the main result of this section.
In the statement, we will need the following quantity.

\begin{notation} \label{notat:theta}
For a prime power $q$ and non-negative integers
$u$ and $i$ with $u \le n$ and $2u-n \le i \le u$, we let $\theta_q(n,u,i)$ denote the number of pairs $(U,U')$
of $u$-spaces
 $U,U' \le \F_q^n$ with the property that 
$\dim(U \cap U')=i$. 
\end{notation}

The next lemma gives a closed expression for $\theta_q(n,u,i)$. 

\begin{lemma} \label{lem:theta}
Let $q$, $u$ and $i$ be as in Notation~\ref{notat:theta}.
We have
$$\theta_q(n,u,i) =  \sum_{j=i}^u (-1)^{j-i} q^{\binom{j-i}{2}} \, \qbin{n}{i}{q}
\, \qbin{n-i}{j-i}{q} \, \qbin{n-j}{u-j}{q}^2.$$
\end{lemma}

\begin{proof}
We will use M\"obius inversion in the lattice of subspaces of $\F_q^n$.
For a subspace $W \le \F_q^n$, let $f(W):=|\{(U,U')\mid U,U' \le \F_q^n, \,  \dim(U)=\dim(U')=u, \, U \cap U' = W\}|$.
Observe that for all $W \le \F_q^n$ we have
$$g(W):=\sum_{\substack{L \le \F_q^n \\ L \ge W}}f(L) =  \qbin{n-\dim(W)}{u-\dim(W)}{q}^2.$$
We now use the
Möbius inversion formula for the lattice of subspaces of $\F_q^n$ (see e.g. Proposition~3.7.2 and Example~3.10.2 in~\cite{stanley2011enumerative}), finding that for every
 $W \le \F_q^n$ of dimension $i$
we have 
\begin{align*}
f(W) &= \sum_{j=i}^u (-1)^{j-i} q^{\binom{j-i}{2}} \sum_{\substack{L \ge W \\ \dim(L)=j}} g(L) \\
&= \sum_{j=i}^u (-1)^{j-i} q^{\binom{j-i}{2}}
\, \qbin{n-i}{j-i}{q} \, \qbin{n-j}{u-j}{q}^2.
\end{align*}
The desired expression for $\theta_q(n,u,i)$ can be obtained by summing the previous identity over all
subspaces $W \le \F_q^n$ having dimension $i$.
\end{proof}

Our main result on the density function of MRD codes is the following. It provides an upper bound for the number of MRD codes with given parameters in terms of the quantities~$\nu$ and~$\theta$ defined/computed earlier in the paper (Notation~\ref{notat:nu} and Lemma~\ref{lem:theta}).

\begin{theorem} \label{thm:mrdboundcc}
Suppose $d \ge 2$ and let $k=m(n-d+1)$. We have \begin{equation} \label{thm:mrdboundcc1}
\delta_q(n \times m, k,d) \le 1- \frac{\qbin{n}{d-1}{q}^2 \, \nu_q(mn,k,m(d-1))^2 }{\qbin{mn}{k}{q} \, \sum_{i=0}^{d-1} \nu_q(mn,k,mi) \,  \theta_q(n,d-1,i)}.
\end{equation}
\end{theorem}
\begin{proof}
For $q \in Q$, consider the collection $\mU_q$ of subspaces $U_q \le \F_q^n$ with $\dim(U_q)=d-1 \ge 1$. We follow the notation of Remark~\ref{rem:expl} and let $\mA_q=\{\mat(U_q)\mid U_q \in \mU_q\}$ for all $q \in Q$. Note that 
$|\mA_q| = \qbin{n}{d-1}{q}$
and that the MRD codes $\mC_q \le \mat$ of minimum distance $d$ are precisely the common complements of the spaces in $\mA_q$. Furthermore, 
for $U_q,U_q' \in \mU_q$ we have $$\dim(\mat(U_q)\cap\mat(U_q'))=\dim(\mat(U_q\cap U_q')) = mi$$ for some $i \in \{0,1,...,d-1\}$. Therefore,
\begin{align*}
    \sum_{\ell=0}^{mn-k}  \nu_q(mn,k,\ell) \cdot |\{(A_q,A_q') \in \mA_q^2 \mid \dim(A_q\cap A_q')=\ell\}| 
    = \sum_{i=0}^{d-1} \nu_q(mn,k,mi) \, \theta_q(n,d-1,i).
\end{align*}
The desired bound now immediately follows from Theorem~\ref{thm:boundscc}.
\end{proof}

Experimental results indicate that 
the quantity on the RHS of \eqref{thm:mrdboundcc1} is asymptotically $\smash{q^{-(d-1)(n-d+1)+1}}$ as $q \to +\infty$. Since this asymptotic estimate does not seem immediate to derive, we will obtain the sparseness of MRD codes using the concept of an asymptotic partial spread we introduced in Definition~\ref{def:asyspread}.

\begin{definition}
The \textbf{ball} of radius $0 \le r \le n$ in $\mat$ is the set of matrices $M \in \mat$ with $\rk(M) \le r$.
It is well-known that its size 
is
\begin{equation} \label{asball}
    \bbq{n \times m, r}:= \sum_{i=0}^r\qbin{n}{i}{q} \prod_{j=0}^{i-1}(q^m-q^j)\sim q^{r(m+n-r)} \quad \mbox{as $q \to +\infty$}.
\end{equation}
\end{definition}

The following result computes the asymptotic density of MRD codes as $q \to +\infty$ for all parameter sets, showing that they are (very) sparse whenever $n \ge 3$ and $d \ge 2$. This solves Problem~\ref{pb:b}.

\begin{theorem} \label{sparseness}
We have
$$\delta_q(n \times m, m(n-d+1), d) \in O\left(q^{-(d-1)(n-d+1)+1}\right) \quad \mbox{as $q \to +\infty$}.$$
Moreover, 
$$\lim_{q \to +\infty} \delta_q(n \times m, m(n-d+1),d) =
\begin{cases} 
1 &\mbox{if $d=1$,} \\
\sum_{i=0}^m \frac{(-1)^i}{i!} & \mbox{if $n=d=2$,} \\
0 & \mbox{otherwise}. 
\end{cases}$$
\end{theorem}

\begin{proof}
The statement immediately follows from the definitions if $d=1$. We henceforth assume $d \ge 2$. For $q \in Q$, denote by $\mA_q$ the family in the proof of Theorem~\ref{thm:mrdboundcc}. We have
    $\smash{|\mA_q| = \qbin{n}{d-1}{q}\sim q^{(d-1)(n-d+1)}}$ as $q \to +\infty$. Since all the spaces in $\mA_q$ have dimension $m(d-1)$ and 
    $\bbq{n \times m, d} \sim q^{(d-1)(m+n-d+1)}$ as
    $q \to +\infty$ by~\eqref{asball}, this shows that
    $(\mA_q)_{q \in Q}$ is an asymptotic partial spread; see Definition~\ref{def:asyspread}. Therefore the first part of the statement follows from Theorem~\ref{thm:asycc}.
In particular, the density limit is zero whenever 
$n \ge 3$ and $d\ge 2$.
The limit for $n=d=2$ has already been computed in \cite[Proposition VII.5]{antrobus2019maximal}.
\end{proof}

\begin{remark}
In~\cite[Theorem 2.4]{gluesing2020sparseness} it was shown that $\delta_q(3 \times 3,3,3) \sim \frac{1}{3}q^{-3}$ as $q \to +\infty$. On the other hand, Theorem~\ref{sparseness} gives that $\delta_q(3 \times 3,3,3) \in O(q^{-1})$ as $q \to +\infty$. This shows that
our asymptotic bound is not sharp in general.
\end{remark}

We now turn to explaining why MDS and MRD codes behave so differently with respect to density properties.

\begin{remark} \label{rem:dive}
The approach developed in this paper offers an explanation for why MDS and MRD codes exhibit 
different behaviors with respect to sparseness and density.
Recall that, for $1 \le k \le n$, a \textbf{$k$-MDS code} is a $k$-dimensional subspace $C \le \F_q^n$ that 
does not contain any non-zero vector of Hamming weight strictly smaller than 
$n-k+1$; see~\cite[Chapter 11]{macwilliams1977theory}. For a subset $S \subseteq \{1,...,n\}$, let $\F_q^n(S) \le \F_q^n$ denote the space of vectors $x \in \F_q^n$ with $x_i=0$ for all $i \notin S$.
Then $k$-MDS codes can be seen as the common complements of the spaces of the form $\F_q^n(S)$, where $S \subseteq \{1,...,n\}$ has size $n-k$. The number of such spaces is~$\smash{\binom{n}{k}}$, which is negligible with respect to $q$ as $q \to +\infty$.  We can therefore use Theorem~\ref{thm:asycc} to explain why MDS codes are dense as the field size tends to infinity: They are the common complements of a  collection of subspaces whose cardinality is negligible with respect to the field size.

For the case of MRD codes the situation is exactly the opposite. As Remark~\ref{rem:expl} shows, MRD codes are the common complements of a collection of subspaces that form an asymptotic partial spread and whose cardinality, for $n \ge 3$ and $d\ge 2$, is far from being negligible with respect to the field size $q$ as $q \to +\infty$.
In particular, they must be sparse by Theorem~\ref{thm:asycc}.
\end{remark}

Combining Theorem~\ref{thm:boundsv}, Theorem~\ref{thm:asyv}, and the estimate in~\eqref{asball} we can also study the density function of rank-metric codes for any minimum distance and any dimension (not just MRD codes).
In the next result we give upper and lower bounds for this function and their
asymptotic versions as $q$ tends to infinity.
For the case of MRD codes, we find that the bounds one obtains are in general worse than the ones given in Theorem~\ref{thm:mrdboundcc}; see Figure~\ref{fig:comparison}, which reflects the general behavior we observed.

\begin{theorem} \label{thm:mrdboundv}
For all $2 \le d \le n$ and $1\le k \le mn$ we have

\smallskip

\begin{align*}
1-\delta_q(n \times m, k, d) &\le \displaystyle \frac{\displaystyle
(\bbq{n \times m,d-1}-1) \,  \qbin{mn-1}{k-1}{q}}{(q-1) \, \qbin{mn}{k}{q}}, \\[0.3cm]
\delta_q(n \times m, k, d) &\le 1 - \displaystyle\frac{\displaystyle\left(\bbq{n \times m,d-1}-1\right)\,\displaystyle\left(\frac{q^{k}-1}{q^{mn}-1}\right)}{(q-1)+\displaystyle\left(\bbq{n \times m, d-1}-q\right)\displaystyle\left(\frac{q^{k-1}-1}{q^{mn-1}-1}\right)}.
\end{align*}
In particular, 
\begin{align*}
1- \delta_q(n \times m, k,d) &\in O\left(q^{(d-1)(m+n-d+1)-(mn-k)-1}\right) \quad \mbox{as $q \to +\infty$}, \\
   \delta_q(n \times m, k,d) &\in
        O\left(q^{-(d-1)(m+n-d+1)+(mn-k)+1}\right) \quad \mbox{as $q \to +\infty$}.
\end{align*}
Therefore,
$$\lim_{q \to + \infty} \delta_q(n \times m,k,d) = 
\begin{cases} 1 & \mbox{if  $(d-1)(m+n-d+1)\le mn-k$}, \\
0 & \mbox{if $(d-1)(m+n-d+1) \ge mn-k+2$}.
\end{cases}$$
\end{theorem}

\begin{remark}
Theorem~\ref{thm:mrdboundv} does not extend to the case 
$k= mn-(d-1)(m+n-d+1)+1$. Proposition~\ref{prop:onehalf} shows that, for this value of $k$,
$\limsup_{q \to + \infty}\, \delta_q(n \times m,k,d) \le 1/2$. On the other hand, in~\cite[Corollary VII.5]{antrobus2019maximal} the asymptotic density of $2 \times m$ MRD codes of dimension~$m$ was computed as $q \to +\infty$, proving that
$\lim_{q \to + \infty}\, \delta_q(2 \times m,m,2) >0$.
This shows that, in general, rank-metric codes of dimension
$k= mn-(d-1)(m+n-d+1)+1$ are neither sparse, nor dense, as
$q \to +\infty$.
\end{remark}

\begin{figure}[h]
\centering
\begin{tikzpicture}[scale=1]
\begin{axis}[legend style={at={(0.93,0.9)}, anchor = north east},
		legend cell align={left},
		width=13cm,height=8cm,
    xlabel={Values of $q$},
    xmin=0, xmax=30,
    ymin=0, ymax=0.15,
    xtick={2,5,10,15, 20, 25, 30},
    ytick={0,0.05,0.1,0.15},
    ymajorgrids=true,
    grid style=dashed,
     every axis plot/.append style={thick},  yticklabel style={/pgf/number format/fixed}
]
\addplot+[color=blueish,mark=square,mark size=1pt,smooth]
coordinates {
    (2,0.110917961672382372867113769688)
(3,0.129026524494521861972425400496)
(4,0.123287258101497510626354953227)
(5,0.113475125748967373475943829331)
(7,0.0949868052569529395383853392487)
(8,0.0873440279264669317463461195116)
(9,0.0807061788753077933728809956256)
(11,0.0698812019291066937646585898411)
(13,0.0615121742790498775027136481521)
(16,0.0520712797956251489065881816296)
(17,0.0495260205064709141556997934105)
(19,0.0451065014616512900696190237045)
(23,0.0382578452296056883148833987318)
(25,0.0355533609036389342071944996194)
(27,0.0332039884080708430941015137531)
(29,0.0311444801970344273912271817675)
(31,0.0293245576879204639258555479072)
(32,0.0284918023066196289518691971056)
(37,0.0249475447073089044372455903977)
(41,0.0226882772507479364107682924123)
(43,0.0217051592819086336624248906469)
(47,0.0199737584836556711599266246525)
(49,0.0192075229464873626528893035237)
(53,0.0178386473238730394651386686111)
(59,0.0161155104602309999278519303421)
(61,0.0156127347873415402698399622280)
(64,0.0149147166805079037212097668207)
(67,0.0142763970778614240416924910323)
(71,0.0135056544592087888585783556855)
(73,0.0131506515711527873168813146328)
(79,0.0121893801728231316726658302936)
(81,0.0118994275259788672820919615392)
(83,0.0116229423468179515418883143237)
(89,0.0108655237101049089579353202531)
(97,0.00999686513235324424957728962400)
(101,0.00961260250638997750741803913018)
(103,0.00943133680406044249694859798003)
(107,0.00908856334209349183634430937190)
(109,0.00892635071895234763994106395382)
(113,0.00861869357809763500949903277923)
(121,0.00806288603166586504264702085815)
(125,0.00781101965404132382861496542575)
(127,0.00769089537698515544714448781954)
(128,0.00763220792691554033669191973313)
(131,0.00746139830136190386563286096207)
(137,0.00714172911239226688464071078789)
(139,0.00704117296683542459782374398106)
(149,0.00657806816096686550409247882578)
(151,0.00649266141694111386929212875347)
(157,0.00624924728360864242949829551898)
(163,0.00602342304585894255913745122325)
(167,0.00588172641348946503733044674041)
(169,0.00581334874551831058609704188356)
(173,0.00568125407000670788951667708926)
(179,0.00549399577563675400632752450667)
(181,0.00543428946078262825381778947833)
(191,0.00515421894236855064075705618603)
(193,0.00510163341115962933996943628857)
(197,0.00499961672555037625316292077868)
(199,0.00495012312806265509380981986570)
(211,0.00467258477960074511303756582570)
(223,0.00442451380308435762359726528824)
(227,0.00434757479874080696387989169492)
(229,0.00431010002261011611583660070722)
(233,0.00423705564307970276059427408265)
(239,0.00413201587941784999638385362102)
(241,0.00409815041791529222693842932143)
(243,0.00406483552873366880280239726701)
(251,0.00393682163445370237918094434859)
(256,0.00386082825913724901573289941538)
(257,0.00384598027397562541926687598682)
(263,0.00375923646820818143786844437657)
(269,0.00367631910290115649165558878384)
(271,0.00364948686162202119395644909044)
(277,0.00357128976696743556256520274642)
(281,0.00352099376815540480851847172176)
(283,0.00349637328871753965130588794408)
(289,0.00342453522763948063362630354704)
(293,0.00337826098722291962116892967684)
(307,0.00322570432580123541761051935410)
(311,0.00318461512103641758171318045223)
(313,0.00316446056957754626939429644288)
(317,0.00312490719332944222052660959087)
(331,0.00299393040406706063198428011873)
(337,0.00294109915921963440034576443281)
(343,0.00289010006802226734087834404919)
(347,0.00285707201095403451879288930424)
(349,0.00284083945041729520997973940462)
(353,0.00280892145409655273037957569362)
(359,0.00276236693403414567793480870136)
(361,0.00274718979547871488581110087228)
(367,0.00270264277571557358052692362108)
(373,0.00265951737516182122744343134764)
(379,0.00261774661233626556838218347065)
(383,0.00259062082139140823349231837270)
(389,0.00255097004862185107221178771667)
(397,0.00249995261255776943698568362733)
(401,0.00247520153593680909486813842905)
(409,0.00242714111357667042412986700185)
(419,0.00236962791359947448411151575612)
(421,0.00235845080319087697749715152784)
(431,0.00230411039711934610219226372627)
(433,0.00229354142279672074603117077446)
(439,0.00226240835352808704779000585119)
(443,0.00224211831956443106821471694169)
(449,0.00221235657984337731470083114267)
(457,0.00217388192976223197756942007267)
(461,0.00215514210030269663377497934686)
(463,0.00214589282312738493486200045569)
(467,0.00212763041066129815142096060526)
(479,0.00207466176354803373078736465060)
(487,0.00204079059976922432028706484338)
(491,0.00202426639291772842642728637858)
(499,0.00199200795203174425523610219529)
    };
\addplot+[color=blush,mark=o,mark size=1pt,smooth]
coordinates {
(2,0.0571193301732408874395914763422)
(3,0.0895860668080219647948805558672)
(4,0.0953323140519103109067978972788)
(5,0.0928644679506458497208190524940)
(7,0.0825332080519224762001780169278)
(8,0.0772676826812215699646200554629)
(9,0.0723869150502000015002611845672)
(11,0.0639341870377944451166314404035)
(13,0.0570502788766965857820123586303)
(16,0.0489761737067536979834143816346)
(17,0.0467488888313594671315205165321)
(19,0.0428341487275661621532089716759)
(23,0.0366554274672771814090414899035)
(25,0.0341798315315725843576977232536)
(27,0.0320135785503349175054222753384)
(29,0.0301028601219269701599294228932)
(31,0.0284054735364540134015155071526)
(32,0.0276260300639554077271909231605)
(37,0.0242896639087979256612642234895)
(41,0.0221471600842966973876112251063)
(43,0.0212110947463994829435057385557)
(47,0.0195570960662517410121835988673)
(49,0.0188229110160594295305091030838)
(53,0.0175079618312857230253003128779)
(59,0.0158466946271321720476490688671)
(61,0.0153607198198427683670538729842)
(64,0.0146850964699387406765257697186)
(67,0.0140663137722623290884116389405)
(71,0.0133179673168924317204218105820)
(73,0.0129728426576969132485176064591)
(79,0.0120369414687896652093042086650)
(81,0.0117542477855727931225187600918)
(83,0.0114845151332527742814959471403)
(89,0.0107447509400876366837743762818)
(97,0.00989482443115178964938243290150)
(101,0.00951833412387864070722549767389)
(103,0.00934062571378453977239436410808)
(107,0.00900438832848733603786175341499)
(109,0.00884518195212466318505334676317)
(113,0.00854307354715539888126518258719)
(121,0.00799678347865295923659150999415)
(125,0.00774901571827163926457049956140)
(127,0.00763079925314918584165500070755)
(128,0.00757303285665540512906647932377)
(131,0.00740486282277226340072717838197)
(137,0.00708996929858142476477517398221)
(139,0.00699087117509322185418742690352)
(149,0.00653420852536289415620825340776)
(151,0.00644994099213207204916345887116)
(157,0.00620969033199550796054013889158)
(163,0.00598669075945925782309922944346)
(167,0.00584671241553694095017340103533)
(169,0.00577914902632435660059285258168)
(173,0.00564859994893854982406928490207)
(179,0.00546347075052124352862120542926)
(181,0.00540442802299524484566994402407)
(191,0.00512737186724300723448827673795)
(193,0.00507533423540796054396073447408)
(197,0.00497436420365012990461432225254)
(199,0.00492537065600248572754089434279)
(211,0.00465054281407299559786962545505)
(223,0.00440476032827728132488219245062)
(227,0.00432850538773015983095231019465)
(229,0.00429135939450560153429952353564)
(233,0.00421894757557383969635665470386)
(239,0.00411479824984061314497535165681)
(241,0.00408121504452950443849593966249)
(243,0.00404817552724323445200858311974)
(251,0.00392119858833418721793315781521)
(256,0.00384580489871618001467702886613)
(257,0.00383107270181684057406078826161)
(263,0.00374499632318067706688508636270)
(269,0.00366270254587417552838819338122)
(271,0.00363606908643407891705693584886)
(277,0.00355844290221348419630451496273)
(281,0.00350850750361569516091584092640)
(283,0.00348406165723647141730995713559)
(289,0.00341272606422253896241263729514)
(293,0.00336676990476471338219845525580)
(307,0.00321523092859652520600883716308)
(311,0.00317440770671297561135302157157)
(313,0.00315438236302388825387113785807)
(317,0.00311508014961566093914250423044)
(331,0.00298491229595637210606565352241)
(337,0.00293239745391471023108053473474)
(343,0.00288169840134155736164389624317)
(347,0.00284886183052824106028709882361)
(349,0.00283272256749298193980699784261)
(353,0.00280098645776986132655947722808)
(359,0.00275469351492573330280019352179)
(361,0.00273960069965446734995238193745)
(367,0.00269529847504279054829793750140)
(373,0.00265240621362056919568564131613)
(379,0.00261085766233979786269163119642)
(383,0.00258387427574385662008901298360)
(389,0.00254442897076417205179507638243)
(397,0.00249367120562373342865770498580)
(401,0.00246904420388188121295341131594)
(409,0.00242122115155158293504473558185)
(419,0.00236398584642327693960183809201)
(421,0.00235286196306242367705675564757)
(431,0.00229877672249791061924462292878)
(433,0.00228825668096905856566601944075)
(439,0.00225726643668265174571386384958)
(443,0.00223706842668063721761598018560)
(449,0.00220744015905242640195702630238)
(457,0.00216913539478285819194374622770)
(461,0.00215047722499608757989850736282)
(463,0.00214126798966828696681343184732)
(467,0.00212308412956419474127816821673)
(479,0.00207033949526118860216260319585)
(487,0.00203660859913653282856910165768)
(491,0.00202015197933892601511703602272)
(499,0.00198802388831898008657551306425)
};
\legend{\small{Theorem~\ref{thm:mrdboundv}}, \small{Theorem~\ref{thm:mrdboundcc}}}
\end{axis}
\end{tikzpicture}
\caption{\label{fig:comparison} Comparison of the upper bounds for $\delta_q(3\times 5, 5, 3)$ as $q \to +\infty$ in Theorem~\ref{thm:mrdboundcc} (red) and in Theorem~\ref{thm:mrdboundv} (blue).}
\end{figure}
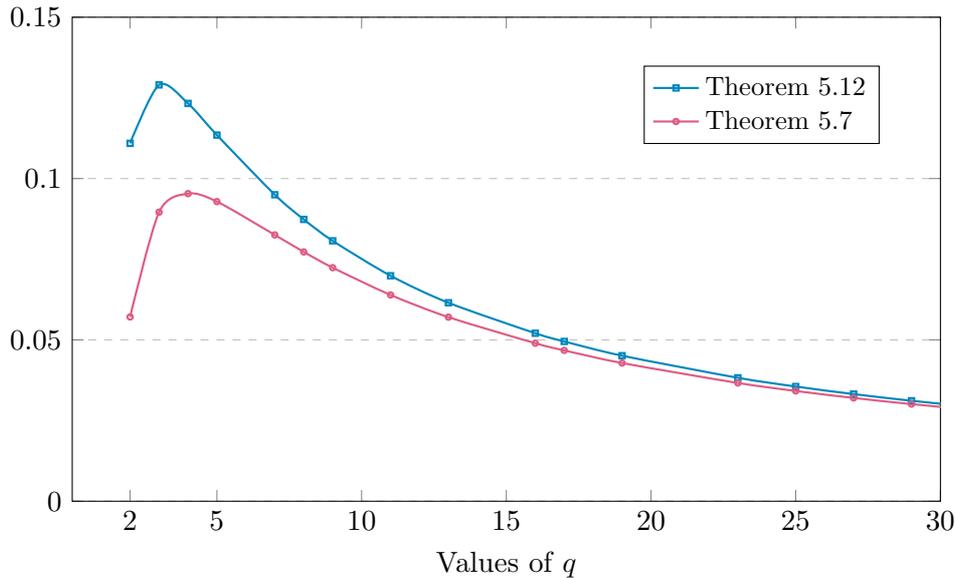

\bigskip

\section{Asymptotic Density of MRD Codes for $m \to +\infty$} \label{sec:m}

In this section we study the asymptotic density of MRD codes as their number of columns, namely $m$, tends to infinity.
Although our approach is not powerful enough to compute the ``exact'' asymptotic density in this setting, as we will see
it improves on known results for several parameter sets.

\begin{notation}
In the sequel we fix a prime power $q$ and integers $n$, $d$ with $n \ge 2$ and $n \ge d \ge 1$. We omit ``$m \in \N, \; m \ge n$'' when writing $m \to +\infty$.
\end{notation}

As for Section~\ref{sec:rankmetric}, we start by surveying the previous literature. The analogue of Theorem~\ref{thm:ravby} 
for $m \to + \infty$
is the following.

\begin{theorem}[\text{see \cite[Corollary 6.4]{byrne2020partition}}] For all $d \ge 2$ we have
\begin{align*}
    \limsup_{m \to +\infty}\, \delta_q(n \times m, m(n-d+1), d) \le \frac{(q-1)(q-2)+1}{2(q-1)^2} \le \frac{1}{2}.
\end{align*}
\end{theorem}

The analogue of Theorem~\ref{thm:heide} for $m \to +\infty$ is~\cite[Theorem VII.6]{antrobus2019maximal}, which we directly state in the
language of this paper for convenience.
The equivalence with~\cite[Theorem VII.6]{antrobus2019maximal} easily
follows from the estimate in~\eqref{eq:pias} below and~\cite[Theorem VII.1]{antrobus2019maximal}.

\begin{theorem}[\text{see \cite[Theorem VII.6]{antrobus2019maximal}}]
\label{jhm}
For all $d \ge 2$ we have
\begin{align*}
    \limsup_{m \to +\infty} \, \delta_{q}(n \times m, m(n-d+1), d) \le \prod_{i=1}^{\infty} \left(1-\frac{1}{q^i}\right)^{q(d-1)(n-d+1)+1}.
\end{align*}
\end{theorem}

In order to derive the asymptotic version of Theorem~\ref{thm:mrdboundcc} for $m \to +\infty$, we will first compute the asymptotics of the quantities 
it involves (Lemma~\ref{lem:num} below).
 For ease of notation, let
 \begin{equation} \label{not:gamma}
\pi(q) := \prod_{i=1}^{\infty}\displaystyle\left(\frac{q^i}{q^i-1}\right).
\end{equation}
The quantity $\pi(q)$ arises in the asymptotic estimate of the $q$-binomial coefficient $\qbin{ma}{mb}{q}$ as $m$ tends to infinity. More precisely, for all integers $a>b>0$ we have
\begin{align} \label{eq:pias}
    \qbin{ma}{mb}{q} = \prod_{i=0}^{mb-1}\frac{\left(q^{ma}-q^i\right)}{\left(q^{mb}-q^i\right)}= q^{mb(ma-mb)} \prod_{i=1}^{mb}\frac{1-q^{i-ma-1}}{1-{q^{-i}}} \sim q^{m^2b(a-b)}\,\pi(q)
\end{align}
as $m\to +\infty$.
We will need this estimate later.

 \begin{remark} \label{rem:euler}
 The infinite product $\pi(q)$ in~\eqref{not:gamma} is closely related to the 
  Euler function $\phi$; see~\cite[Section 14]{apostol2013introduction} for a standard reference. The latter is the function $\phi:(-1,1) \to \R$ defined by
\begin{equation} \label{euler}
\phi(x) := \prod_{i=1}^{\infty}(1-x^i)
 \end{equation}
 for all $x \in (-1,1)$.
 We then have $\pi(q)=1/\phi(1/q)$ for all $q \in Q$. In particular,
 $\pi(q)>1$ for all $q \in Q$.
 A classical result in number theory, due to Euler himself, expresses the infinite product in~\eqref{euler} as the infinite sum
 $$\phi(x) = 1+\sum _{k=1}^{\infty }(-1)^{k}\left(x^{k(3k+1)/2}+x^{k(3k-1)/2}\right) = 1-x-x^{2}+x^{5}+x^{7}-x^{{12}}-x^{{15}}+\ldots.$$
 This is the famous Pentagonal Number Theorem~\cite[Theorem 14.3]{apostol2013introduction}. From the above expression for $\phi$ as a power series, we deduce the following asymptotic estimates:
 \begin{equation} \label{eq:euler}
     \phi(x) \sim 1 \mbox{ as $x \to 0$}, \qquad \quad \phi(x)-1 \sim -x \mbox{ as $x \to 0$}.
 \end{equation}
 \end{remark}

The next result gives an asymptotic estimate for $\nu_q(mn,m(n-d+1),mi)$ for $1 \le i \le d-1$ as $m \to +\infty$. We will need it to establish Theorem~\ref{thm:mrdboundccm}.
Note that we only examine dimensions that are a multiple of $m$, as for other dimensions MRD codes do not exist. Furthermore, we will only need multiples of $m$ as intersection dimensions. The proof of the next lemma can be found in the Appendix.

\begin{lemma} \label{lem:num}
Let $d \ge 2$ and let $0 \le i \le d-1$ be an integer. The following  asymptotic estimates hold as $m \to + \infty$: 
\begin{align*}
    \nu_q(mn,m(n-d+1),mi) \sim \begin{cases}
    q^{m^2(n-d+1)(d-1)}\,\left(\pi(q)-1\right)^2/\pi(q) &\textnormal{ if } 0 \le i \le d-2, \\
    q^{m^2(n-d+1)(d-1)}\,\left(\pi(q)-1\right) &\textnormal{ if } i=d-1.
    \end{cases}
\end{align*}
\end{lemma}

We are now ready to derive the asymptotic version of Theorem~\ref{thm:mrdboundcc} as $m \to +\infty$.

\begin{theorem} \label{thm:mrdboundccm}
For all $d \ge 2$ we have 
\begin{equation*}
\limsup_{m \to +\infty} \, \delta_q(n \times m, m(n-d+1),d) \le  \frac{1}{\displaystyle\qbin{n}{d-1}{q}\left(\pi(q)-1\right)+1}.
\end{equation*}
\end{theorem}
\begin{proof}
To simplify the notation throughout the proof, let $k_m=m(n-d+1)$. All estimates in the sequel are for $m \to +\infty$.
The upper bound on $\delta_q(n \times m, k_m,d)$ given in Theorem~\ref{thm:mrdboundcc} reads
\begin{equation}\label{initial}
    \delta_q(n \times m, k_m, d) \le \frac{a_m+b_m-c_m}{a_m+b_m},
\end{equation}
where:
\begin{align*}
    a_m &= \textstyle\sum_{i=0}^{d-2} \nu_q(mn,k_m,mi) \,  \theta_q(n,d-1,i), \\
    b_m &= \theta_q(n,d-1,d-1) \, \nu_q(mn,k_m,m(d-1)), \\
    c_m &= \frac{\qbin{n}{d-1}{q}^2 \, \nu_q(mn,k_m,m(d-1))^2}{\qbin{mn}{k_m}{q}}.
\end{align*}
The three quantities above can be conveniently estimated individually with the aid of 
Lemma~\ref{lem:num}.
All the computations are tedious but straightforward, so we only include the final results:
\begin{equation}\label{eq:aprojspb} 
\left\{ \; \begin{split}
    	a_m &\sim  \left(\pi(q)-1\right)^2/\pi(q)  \, \textstyle\sum_{i=0}^{d-2} \theta_q(n,d-1,i) \, q^{k_m(mn-k_m)},\\
	 b_m &\sim \theta_q(n,d-1,d-1) \, (\pi(q)-1) \, q^{k_m(mn-k_m)}, \\
	 c_m &\sim \qbin{n}{d-1}{q}^2 \left(\pi(q)-1\right)^2 / \pi(q) \, q^{k_m(mn-k_m)}.
\end{split}\right.
\end{equation}	
Using   Notation~\ref{notat:theta} directly we find
\begin{align*}
    \textstyle\sum_{i=0}^{d-2}  \theta_q(n,d-1,i) &= \qbin{n}{d-1}{q}\,(\qbin{n}{d-1}{q}-1), \\
    \theta_q(n,d-1,d-1) &= \qbin{n}{d-1}{q}.
\end{align*}
This allows us to rewrite~\eqref{eq:aprojspb} as
\begin{equation}\label{eq:aprojspb2} 
\left\{ \; \begin{split}
    	a_m &\sim  \left(\pi(q)-1\right)^2/\pi(q) \,  \qbin{n}{d-1}{q}\,(\qbin{n}{d-1}{q}-1) \, q^{k_m(mn-k_m)},\\
	 b_m &\sim  (\pi(q)-1) \, \qbin{n}{d-1}{q} \, q^{k_m(mn-k_m)}, \\
	 c_m &\sim \qbin{n}{d-1}{q}^2  \left(\pi(q)-1\right)^2 / \pi(q) \, q^{k_m(mn-k_m)}.
\end{split}\right.
\end{equation}	
Now observe that the three estimates in~\eqref{eq:aprojspb2} are of the form 
$a_m \sim a f_m$, 
$b_m \sim b f_m$ and
$c_m \sim c f_m$,
where $f_m=q^{k_m(mn-k_m)}$
and $a,b,c \in \R$ are positive constants in $m$. Moreover, one can check that
$$\frac{a+b-c}{a+b} = \frac{1}{\displaystyle\qbin{n}{d-1}{q}\left(\pi(q)-1\right)+1} \neq 0.$$
Therefore the desired theorem  follows by taking the limit superior as $m \to +\infty$ in~\eqref{initial}. 
\end{proof}

As for the case where
$q \to +\infty$, the bound on the density of MRD codes that one obtains from Theorem~\ref{thm:mrdboundcc} is better than the one from Theorem~\ref{thm:boundsv}. We elaborate on this in the following remark.

\begin{remark} For $m$ sufficiently large, the bound on $\delta_q(n \times m, m(n-d+1),d)$ obtained from Theorem~\ref{thm:boundsv} is worse than the one of Theorem~\ref{thm:mrdboundcc} (which we in turn derived from Theorem~\ref{thm:boundscc}). This is true even in the asymptotics.
More precisely,
recall the following estimate for the size of the ball:
\begin{equation} \label{asballm}
    \bbq{n \times m, d-1} \sim \qbin{n}{d-1}{q} q^{m(d-1)} \quad \mbox{as $m \to +\infty$.}
\end{equation}
Using this estimate, one can obtain the following asymptotic version of the bound of Theorem~\ref{thm:mrdboundv}:
\begin{align} \label{eq:mrdboundvm}
\limsup_{m \to +\infty}\,\delta_q(n \times m, m(n-d+1), d) &\le \frac{q-1}{\qbin{n}{d-1}{q}+q-1} \mbox{ whenever $d \ge 2$}.
\end{align}
An easy computation shows that the upper bound in Theorem~\ref{thm:mrdboundccm} is always sharper than the one in~\eqref{eq:mrdboundvm}.
\end{remark}

\begin{remark}
The asymptotic upper bound of Theorem~\ref{thm:mrdboundccm} is sharper than the bound of~\cite[Theorem VII.6]{antrobus2019maximal} for $q$ sufficiently large, $n \ge d \ge 2$ and $n >2$, while it is coarser for small values of $q$. To show this, we first rewrite~\cite[Theorem VII.6]{antrobus2019maximal}
(in the form stated in Theorem~\ref{jhm}) as \begin{align*}
    \limsup_{m \to +\infty} \, \delta_{q}(n \times m, m(n-d+1), d) \le \frac{1}{\pi(q)^{q(d-1)(n-d+1)+1}}.
\end{align*}
We now prove that
\begin{equation} \label{heideandus}
\lim_{q \to +\infty} \frac{\pi(q)^{q(d-1)(n-d+1)+1}}{\qbin{n}{d-1}{q}(\pi(q)-1)+1} = 0 \quad \mbox{for $n \ge d \ge 2$ and $n >2$},
\end{equation}
from which it immediately follows that
the bound of  Theorem~\ref{thm:mrdboundccm} is sharper than the one of
~\cite[Theorem VII.6]{antrobus2019maximal} for $q$ sufficiently large.
To see why~\eqref{heideandus} holds, we note first that
\begin{align*}
    \lim_{q \to +\infty} \left(1-\frac{1}{q^i}\right)^q = \begin{cases}
    1/e \quad &\textnormal{ if } i=1, \\
    1 \quad &\textnormal{ if } i \ge 2.
    \end{cases}
\end{align*}
Therefore $\lim_{q \to +\infty}\phi(1/q)^q=\prod_{i=1}^{\infty} \left(1-1/q^i\right)^q = 1/e$. Since $\pi(q)=1/\phi(1/q)$, as already observed in Remark~\ref{rem:euler}, we have $\lim_{q \to +\infty}\pi(q)^q = e$. Furthermore, using the asymptotic estimates in~\eqref{eq:euler} we find that
\begin{equation*}
\pi(q)-1 = \frac{1-\phi(1/q)}{\phi(1/q)} \sim \frac{1}{q} \quad \text{ as } q \to +\infty.
\end{equation*}
Combining this with the estimate for $\pi(q)^q$ given above we obtain, for $n \ge d \ge 2$, 
$$ \frac{\pi(q)^{q(d-1)(n-d+1)+1}}{\qbin{n}{d-1}{q}(\pi(q)-1)+1} \sim \frac{e^{(d-1)(n-d+1)+1}}{q^{(d-1)(n-d+1)-1}} \quad \textnormal{ as } q \to +\infty.$$ The fraction on the RHS of the previous estimate approaches $0$ as $q$ approaches~$+\infty$, thereby establishing the desired limit in~\eqref{heideandus}.

A comparison of the two bounds we just discussed can be seen in Figure~\ref{fig:comparison2}. The plot 
shows that for $(n,d)=(3,3)$ the bound of Theorem~\ref{thm:mrdboundccm} is sharper than~\cite[Theorem~VII.6]{antrobus2019maximal} for all prime powers $q\ge 9$, and coarser for $2 \le q \le 8$. The two bounds are decreasing and increasing in~$q$, respectively.
\end{remark}

\begin{figure}[h]
\centering
\begin{tikzpicture}[scale=1]
\begin{axis}[legend style={at={(0.04,0.93)}, anchor = north west},
		legend cell align={left},
		width=13cm,height=8cm,
    xlabel={Values of $q$},
    xmin=0, xmax=30,
    ymin=0, ymax=0.15,
    xtick={2,5,10,15, 20, 25, 30},
    ytick={0,0.05,0.1,0.15},
    ymajorgrids=true,
    grid style=dashed,
     every axis plot/.append style={thick},  yticklabel style={/pgf/number format/fixed}
]
\addplot+[color=blueish,mark=square,mark size=1pt,smooth]
coordinates {
(2,0.00200861374478225826545433309169)
(3,0.0172981854536559066298775195858)
(4,0.0347815230728943907748000355616)
(5,0.0490954176866865783402860880426)
(7,0.0690699779444482735868732671831)
(8,0.0760978586138095612708591476479)
(9,0.0818190922629343056556067432747)
(11,0.0905347600108603123388378520459)
(13,0.0968388143804218181044443831483)
(16,0.103568234534981703177034136934)
(17,0.105320312782207408425836493253)
(19,0.108304666163287583548901598085)
(23,0.112792835431191270163985488899)
(25,0.114522132827022491158795225920)
(27,0.116005522417876548095747503820)
(29,0.117291878917620425909980354413)
(31,0.118417952624045207034154186244)
    };
\addplot+[color=blush,mark=o,mark size=1pt,smooth]
coordinates {
(2,0.0548268869286448122598706315360)
(3,0.0892135433490818745938532625822)
(4,0.0952432974139241038778555389371)
(5,0.0928365444429189115539652303954)
(7,0.0825286196622346390460484836825)
(8,0.0772654759364910868946547174941)
(9,0.0723857650525810489961379827504)
(11,0.0639338125913387247111134533627)
(13,0.0570501331875791046278912090352)
(16,0.0489761291228975518106363106458)
(17,0.0467488573425071574213321250036)
(19,0.0428341321257948306765597390756)
(23,0.0366554219710014133047488284911)
(25,0.0341798281459777455411237693009)
(27,0.0320135763878174024613938609447)
(29,0.0301028586968346119618505603765)
(31,0.0284054725714775897626415139338)
};
\legend{\small{Theorem VII.6 in \cite{antrobus2019maximal}}, \small{Theorem~\ref{thm:mrdboundccm}}}
\end{axis}
\end{tikzpicture}
\caption{\label{fig:comparison2} The upper bounds for $\limsup_{m \to +\infty}\delta_q(3\times m, m, 3)$ from Theorem~\ref{thm:mrdboundccm} (red) and~\cite[Theorem VII.6]{antrobus2019maximal} (blue).}
\end{figure}

\bigskip

\section{Further Properties of Density Functions} \label{sec:prop}
We devote the last section of the paper to general properties of the density function of rank-metric codes. More precisely, we initiate the study of 
how 
density functions relate to each other as the  parameters $(n,m,d,k)$ change.
As an application of our results, we reinterpret 
the bounds of \cite{antrobus2019maximal}
via shortening and duality considerations.

\begin{notation}
In the sequel, $q$ is a prime power and $n$, $m$, $d$  denote positive integers with $m \ge n \ge 2$ and $1 \le d \le n$. When writing ``$q \to +\infty$'' or ``$m \to +\infty$'', the other parameters are treated as constants. As in Section~\ref{sec:asy}, the limit for $q \to +\infty$ is taken over the set of all prime powers, denoted by $Q$. 
\end{notation}

\begin{lemma} \label{lem:basis}
Every MRD code $\mC \le \mat$ with minimum distance $d \ge 2$ admits a unique basis of the form
\begin{align}\label{specialbasis}
    \left\{ \left( \begin{array}{c}
    E_{ij} \\
    \hline
    A_{ij}
    \end{array} \right) 
    \mid 1 \le i \le n-d+1, \; 1 \le j \le m \right\},
\end{align}
where $\smash{A_{ij} \in \mathbb{F}_q^{(d-1)\times m}}$ is a suitable matrix and $\smash{E_{ij} \in \mathbb{F}_q^{(n-d+1) \times m}}$ denotes the matrix having a~1 in position $(i,j)$ and 0 elsewhere.
Moreover, all the matrices in~\eqref{specialbasis} have rank exactly $d$.
\end{lemma}
\begin{proof}
The desired lemma follows from the fact that the projection $\smash{\pi: \mC \longrightarrow \F_q^{(n-d+1)\times m}}$
on the first $n-d+1$ rows
is an isomorphism.
Injectivity is a consequence of the fact that $\mC$ has minimum distance $d$ and bijectivity follows from 
cardinality considerations. The last part of the statement can be seen by observing that $\mC$ has minimum distance $d$ and that all matrices in~\eqref{specialbasis} have $n-d$ zero rows.
\end{proof}

\begin{proposition} \label{prop:decomposition}
Suppose $n \ge 3$ and $2 \le d<n$. We have
$$\delta_q(n\times m, m(n-d+1), d) \le \delta_q(d\times m, m, d) \; \delta_q((n-1)\times m, m(n-d), d)\; \frac {B_q(n \times m, d)}{A_q(n \times m, d)},$$
where
\begin{align*}
    A_q(n \times m, d) &= \qbin{mn}{m(n-d+1)}{q}, \\
    B_q(n \times m, d) &=  \qbin{md}{m}{q}\, \qbin{m(n-1)}{m(n-d)}{q}.
\end{align*}
\end{proposition}
\begin{proof}
We start by showing that every MRD code $\mC \le \mat$ with minimum distance $d$ can be decomposed as (the embedding of) a direct sum of an MRD code in $\smash{\F_q^{d\times m}}$ and an MRD code in $\smash{\F_q^{(n-1)\times m}}$, both of which have minimum distance exactly $d$.
For this, let $\mC \le \mat$ be an MRD code with minimum distance~$d$. By Lemma~\ref{lem:basis}, $\mC$ has a basis of the form
\begin{align*}
    \left\{ \left( \begin{array}{c}
    E_{ij} \\
    \hline
    A_{ij}
    \end{array} \right) 
    \mid 1 \le i \le n-d+1, \; 1 \le j \le m \right\},
\end{align*}
with $A_{ij}$ and $E_{ij}$ as in the statement of the lemma. We let $\mC_1$ be the code with basis 
\begin{align*}
    \left\{ \left( \begin{array}{c}
    E_{1j} \\
    \hline
    A_{1j}
    \end{array} \right) 
    \mid 1 \le j \le m \right\}.
\end{align*}
It is easy to see that $\mC_1$ is (after a suitable embedding) an MRD code in $\mathbb{F}_q^{d\times m}$ with minimum distance~$d$. Furthermore, we have that the code $\mC_2$ generated by the basis 
\begin{align*}
    \left\{ \left( \begin{array}{c}
    E_{ij} \\
    \hline
    A_{ij}
    \end{array} \right) 
    \mid 2 \le i \le n-d+1, \,  1 \le j \le m \right\}
\end{align*}
is (again after embedding) an MRD code in $\F_q^{(n-1)\times m}$, which also has minimum distance $d$ by the second part of Lemma~\ref{lem:basis}. It is not hard to see that the mapping $\mC \mapsto (\mC_1,\mC_2)$ is injective, from which we obtain
\begin{multline*}
|\{\mC \le \mat \mid \drk(\mC)=d, \, \mC \textnormal{ is MRD}\}| \le  \\
|\{\mC \le \F_q^{d \times m} \mid \drk(\mC)=d, \, \mC \textnormal{ is MRD}\}| \cdot |\{\mC \le \F_q^{(n-1)\times m} \mid \drk(\mC)=d, \, \mC \textnormal{ is MRD}\}|.
\end{multline*}
Dividing both sides by $A_q(n\times m, d) B_q(n \times m, d)$ yields the desired result.
\end{proof}

It is well-known that the dual of an MRD code in $\mat$ of minimum distance $d$ is an MRD code in $\mat$ of minimum distance $n-d+2$; see e.g.~\cite[Theorem 5.5]{delsarte1978bilinear}.
Since the map that sends a code to its dual is a bijection,
this simple fact can be rephrased in terms of density functions as follows.

\begin{proposition} \label{prop:dual}
For all $d \ge 2$ we have
$$\delta_q(n \times m, m(n-d+1),d) = \delta_q(n \times m, m(d-1),n-d+2).$$
\end{proposition}

We now combine Propositions~\ref{prop:decomposition} and~\ref{prop:dual} in order to illustrate how the density functions of MRD codes with different parameters behave with respect to each other.

\begin{corollary} \label{cor:specbas}
For all $d \ge 2$ we have
\begin{align*}
    \delta_q(n \times m, m(n-d+1), d) \le \delta_q(2 \times m, m, 2)^{(n-d+1)(d-1)} \; \frac{ \qbin{2m}{m}{q}^{(n-d+1)(d-1)}}{\qbin{mn}{m(n-d+1)}{q}}.
\end{align*}
\end{corollary}
\begin{proof}
By applying Proposition~\ref{prop:decomposition} $n-d+1$ times one obtains
\begin{align} \label{pf:specbas1}
\delta_q(n\times m, m(n-d+1), d) &\le \delta_q(d\times m, m, d)^{(n-d+1)}\, \frac{\qbin{md}{m}{q}^{(n-d+1)}}{\qbin{mn}{m(n-d+1)}{q}}.
\end{align}
By applying the same proposition $d-1$ times one obtains
\begin{align}  \label{pf:specbas2}
\delta_q(d \times m, m(d-1), 2) &\le \delta_q(2 \times m, m, 2)^{(d-1)}\, \frac{\qbin{2m}{m}{q}^{(d-1)}}{\qbin{md}{m(d-1)}{q}}.
\end{align}
Moreover, using Proposition~\ref{prop:dual} we find that $\delta_q(d\times m, m, d) = \delta_q(d \times m, m(d-1), 2)$ which, combined with~\eqref{pf:specbas1} and \eqref{pf:specbas2}, yields the desired result.
\end{proof}

\begin{remark}
Taking the limit superior in Corollary~\ref{cor:specbas} as $q \to +\infty$ and $m \to +\infty$, and combining it with~\cite[Corollary VII.5]{antrobus2019maximal}, we obtain the same asymptotic bounds for the density of MRD codes as~\cite[Theorem VII.6]{antrobus2019maximal}.
The $q$-binomial coefficients can be estimated using~\eqref{eq:qbin} and~\eqref{eq:pias}.
\end{remark}

\section*{Acknowledgement}
The authors are very grateful to the Referees of this paper for their very careful reading of the manuscript and suggestions.

\appendix

\section{Some Proofs}

\begin{proof}[Proof of Lemma~\ref{lem:nuq}]
We establish the two groups of asymptotic estimates separately. Easy computations show that
\begin{align} \label{pf:nuq1}
\frac{\nu_q(N,k,\ell)}{q^{k(N-k)}} = \frac{\displaystyle\prod_{i=N-k+1}^N\left(1- \frac{1}{q^i}\right)-2\displaystyle\prod_{i=1}^{k}\left(1- \frac{1}{q^i}\right)+\displaystyle\prod_{i=1}^{N-k-\ell}\left(1- \frac{1}{q^i}\right)\displaystyle\prod_{i=1}^{k}\left(1- \frac{1}{q^i}\right)}{\displaystyle\prod_{i=1}^{k}\left(1- \frac{1}{q^i}\right)}.
\end{align}
First note that 
$$\displaystyle\prod_{i=1}^{k}\left(1- \frac{1}{q^i}\right) \sim 1 \quad \mbox{ as } q \to +\infty.$$
The different cases are treated separately and for each of them we focus on computing the  asymptotics of the numerator of the right-hand side of~\eqref{pf:nuq1}.

\underline{Case 1}. We first assume $\max\{0,N-2k\} \le \ell < N-k-1$. Note that for $q \to +\infty$ we have
    \begin{align*}
        \displaystyle\prod_{i=N-k+1}^N\left(1- \frac{1}{q^i}\right) &= 1 + O\left(\frac{1}{q^{3}}\right), \\
        -2\displaystyle\prod_{i=1}^{k}\left(1- \frac{1}{q^i}\right) &= -2+\frac{2}{q}+\frac{2}{q^2}+O\left(\frac{1}{q^{3}}\right), \\ \displaystyle\prod_{i=1}^{N-k-\ell}\left(1- \frac{1}{q^i}\right)\displaystyle\prod_{i=1}^{k}\left(1- \frac{1}{q^i}\right) &= 1-\frac{2}{q}-\frac{1}{q^2} + O\left(\frac{1}{q^{3}}\right),
    \end{align*}
    which all together show that the numerator's asymptotic for $q \to +\infty$ is $q^{-2}$.
    
\underline{Case 2}. Next assume that $\ell=N-k-1=\max\{0,N-2k\}$. If $k=N-1$ we have
    \begin{align*}
        \displaystyle\prod_{i=2}^N\left(1- \frac{1}{q^i}\right)-\displaystyle\prod_{i=1}^{N-1}\left(1- \frac{1}{q^i}\right)-\frac{1}{q}\displaystyle\prod_{i=1}^{N-1}\left(1- \frac{1}{q^{i}}\right)  = \frac{1}{q^2} + O\left(\frac{1}{q^3}\right) \quad \mbox{ as } q \to +\infty,
    \end{align*}
    since all terms in the expression which are preponderant with respect to ${q^{-2}}$ vanish.  We leave the case $k=1$ to the reader.

\underline{Case 3}. Assume that $\ell=N-k-1 >  \max\{0,N-2k\}$. The numerator in~\eqref{pf:nuq1} reduces to
    \begin{align*}
        \displaystyle\prod_{i=N-k+1}^N\left(1- \frac{1}{q^i}\right)-\displaystyle\prod_{i=1}^{k}\left(1- \frac{1}{q^i}\right)-\frac{1}{q}\displaystyle\prod_{i=1}^{k}\left(1- \frac{1}{q^{i}}\right)  = \frac{2}{q^2} + O\left(\frac{1}{q^3}\right) \quad \mbox{ as } q \to +\infty,
    \end{align*}
    from which the statement follows.

\underline{Case 4}. Assume that $\ell=N-k$. By convention (Notation~\ref{notat:nu}) we have
\begin{align*}
    \prod_{i=1}^{N-k-\ell}\left(1- \frac{1}{q^i}\right) = 1
\end{align*}
and therefore the numerator of the right-hand side of~\eqref{pf:nuq1} becomes
    $$\displaystyle\prod_{i=N-k+1}^N\left(1- \frac{1}{q^i}\right)-\displaystyle\prod_{i=1}^{k}\left(1- \frac{1}{q^i}\right)  = \frac{1}{q} +  O\left(\frac{1}{q^2}\right) \quad \mbox{ as } q \to +\infty.$$ 
    The desired estimate follows.

In order to prove the second group of asymptotic
estimates, observe first that by Notation~\ref{notat:nu} we have
\begin{align*}
    \frac{\nu_q(N,k,N-k)^2}{\qbin{N}{k}{q}} - \nu_q(N,k,\ell) =  
    \frac{q^{2k(N-k)}}{\qbin{N}{k}{q}}-q^{(2k-N+\ell)(N-k)}\prod_{i=\ell}^{N-k-1}(q^{N-k}-q^i).
\end{align*}
Using the definition of the $q$-binomial coefficient we find that the above expression reduces to
\begin{align*}
q^{k(N-k)}\, \frac{\displaystyle\prod_{i=1}^k\left(1-\frac{1}{q^i}\right)-\displaystyle\prod_{i=1}^{N-k-\ell}\left(1-\frac{1}{q^i}\right)\displaystyle\prod_{i=N-k+1}^N\left(1-\frac{1}{q^i}\right)}{\displaystyle\prod_{i=N-k+1}^N\left(1-\frac{1}{q^i}\right)}.
\end{align*}
The asymptotic estimates can be conveniently derived from this formula, examining the various cases separately.
\end{proof}

\begin{proof}[Proof of Lemma~\ref{lem:num}] By Equation~\eqref{pf:nuq1} we have
\begin{multline} \label{pf:num1}
\nu_q(mn,m(n-d+1),mi) = \\ q^{m(n-d+1)m(d-1)}\, \left( \frac{\displaystyle\prod_{i=m(d-1)+1}^{mn}\left(1- \frac{1}{q^i}\right)}{\displaystyle\prod_{i=1}^{m(n-d+1)}\left(1- \frac{1}{q^i}\right)} -2+\displaystyle\prod_{i=1}^{m(d-1-i)}\left(1- \frac{1}{q^i}\right)\right).
\end{multline}
By the definition of $\pi(q)$ from~\eqref{not:gamma} we have 
\begin{align} \label{pf:num2}
    \frac{\displaystyle\prod_{i=m(d-1)+1}^{mn}\left(1- \frac{1}{q^i}\right)}{\displaystyle\prod_{i=1}^{m(n-d+1)}\left(1- \frac{1}{q^i}\right)} \sim \pi(q) \quad \mbox{ as } m \to +\infty.
\end{align}
Furthermore, as $m \to +\infty$, we  have
\begin{align}\label{pf:num3}
    \displaystyle\prod_{i=1}^{m(d-1-i)}\left(1- \frac{1}{q^i}\right) \sim \begin{cases}
    1 &\textnormal{ if } i=d-1, \\
    {1}/{\pi(q)} &\textnormal{ if } i<d-1.
    \end{cases}
\end{align}
Combining~\eqref{pf:num1},~\eqref{pf:num2} and~\eqref{pf:num3} the desired result follows.
\end{proof}

\ 

\bigskip

\bibliographystyle{amsplain}
\bibliography{ourbib}

\end{document}